\documentclass[10pt]{article}

\usepackage{amssymb,amsmath, graphicx, amsthm,sectsty,stmaryrd,lmodern,titlesec,color,hyperref,mathpazo}
\usepackage[mathscr]{euscript}

\newcommand{\id}{\mathrm{id}}

\newcommand{\Vect}{\mathrm{\textbf{Vect}}}
\newcommand{\<}{\left\langle}
\renewcommand{\>}{\right\rangle}
\newcommand{\tensor}{\otimes}

\newcommand{\Aut}{\mathrm{Aut}}

\renewcommand{\Bar}{\overline}
\renewcommand{\Tilde}{\widetilde}

\renewcommand{\Hat}{\widehat}

\def\bfC{{\bf C}}
\def\bfR{{\bf R}}
\def\bfZ{{\bf Z}}

\newcommand{\bbZ}{\mathbb{Z}}

\newcommand{\scr}{\mathscr}

\newcommand{\ul}{\underline}

\newcommand{\dd}{{\rm d}}
\newcommand\dbar{\Bar{\partial}}
\newcommand\zbar{\Bar{z}}

\def\be{\begin{equation}}
\def\ee{\end{equation}}
\def\bearray{\begin{eqnarray}}
\def\eearray{\end{eqnarray}}
\def\bestar{\begin{eqnarray*}}
\def\eestar{\end{eqnarray*}}
\def\ben{\begin{displaymath}}
\def\een{\end{displaymath}}
\theoremstyle{plain}
\newtheorem{theorem}{Theorem}[section]
\newtheorem{lemma}[theorem]{Lemma}

\newtheorem{prop}[theorem]{Proposition}

\theoremstyle{definition}
\newtheorem{defn}{Definition}[section]

\newtheorem{example}{Example}[section]
\theoremstyle{remark}
\newtheorem*{remark}{Remark}

\usepackage[margin=1in]{geometry}
\usepackage[all,cmtip]{xy}
\newcommand{\xym}{\xymatrix@1@=14pt@M=2pt}

\title{The Virasoro vertex algebra and factorization algebras on
  Riemann surfaces}
\author{Brian Williams \\ Northwestern University}
\begin{document}
\maketitle

\def\Vir{{\rm Vir}}
\def\Vect{{\rm Vect}}
\def\sF{\scr{F}}
\def\sL{\scr{L}}
\def\sE{\scr{E}}
\def\sO{\scr{O}}
\def\VVert{\mathbb{V}{\rm ert}}
\def\brian{\textcolor{blue}{BW: }\textcolor{blue}}
\def\dgNuc{{\rm dgNuc}}
\def\d{{\rm d}}
\def\sVir{\scr{V}{\rm ir}}
\def\oloc{\scr{O}_{\rm loc}}
\def\GL{{\rm GL}}
\def\fg{\mathfrak{g}}
\def\fh{\mathfrak{h}}

\section{Introduction}
In this paper we study the sheaf of holomorphic vector
fields in one complex dimension and local extensions
thereof. Using the formalism of factorization algebras developed in
the book \cite{CG1} we provide a construction of the {\em Virasoro
  factorization algebra} defined on any Riemann surface. Moreover, we
compute and recognize the factorization homology of the two
dimensional factorization
algebra as encoding the conformal blocks of the Virasoro vertex
algebra. 

The Virasoro Lie algebra $\Vir$ arises as a central extension of the
Lie algebra of vector fields on a circle $\Vect(S^1)$. In fact, it is the unique
central extension as ${\rm H}^2(\Vect(S^1))$ is one-dimensional with generator given by the Gelfand-Fuks
cocycle \cite{GF} defined by
\ben
\omega_{\rm GF}(f(t) \partial_t, g(t) \partial_t) \mapsto
\frac{1}{12} \int_{S^1} f'''(t) g(t)\dd t .
\een
The Virasoro Lie algebra, along with its related vertex algebra and
category of representations, are
interesting and natural in their own right from a mathematical point
of view \cite{GF, FF, KW, KI}. 

The compelling motivation for studying of the Virasoro algebra derived from understanding the symmetries of two-dimensional
conformal field theories. Classically, conformal symmetry consists of two copies of the
complexification of the Lie algebra of vector fields on the circle: a
holomorphic and an anti-holomorphic version. We will choose to focus
on holomorphic, or chiral,
conformal field theories and hence only consider holomorphic vector
fields on two-dimensional complex manifolds. The Weyl anomaly arises
when one tries to quantize the symmetry of holomorphic vector fields
on such a conformal field theory. It results in the one-dimensional central
extension of holomorphic vector fields defined by the Gelfand-Fuks
cocycle. Moreover, the anomaly is characterized by how the
central parameter acts on the quantum theory; this is called the
{\em central charge} of the theory.

We work with the Dolbeault resolution of holomorphic vector fields on
$\bfC$, which we denote by $\scr{L}^\bfC$ throughout. The fact that we can restrict vector
fields to open sets gives this the structure of a sheaf of Lie
algebras. Moreover, it has the structure of a {\it local Lie
  algebra} on $\bfC$ which will be central in our construction.  We define an explicit cocycle $\omega$ that defines a $(-1)$-shifted
central extension of this local Lie algebra. There is a factorization
algebra associated to this local Lie
algebra, denoted $\sVir$. We show that the factorization
product encodes the product on the universal enveloping algebra
associated to the ordinary Virasoro Lie algebra, $U(\Vir)$. 

We go further and use a
bcharacterization of structured holomorphic factorization algebras on
$\bfC$ from the book \cite{CG1} to show that this factorization algebra has the structure of a
vertex algebra and it is equivalent to that of the Virasoro
vertex algebra. In \cite{CG1}, a functor $\VVert$ from the
category of structured holomorphic factorization algebras on $\bfC$
to the category of vertex algebras is defined. The main result from the first
part of this paper can be stated as
follows. 

\begin{theorem} 
For any complex number $c \in
  \bfC$ there is a factorization algebra $\sVir_{c}$ on
  $\bfC$ (given
  by the enveloping factorization algebra for the extension of
  $\scr{L}^\bfC$ by the cocycle $c \omega$) which determines a vertex
  algebra $\mathbb{V}{\rm ert}(\sVir_{c})$. Moreover, there is an isomorphism of vertex algebras
\ben\xym{
{\bf Vir}_c \ar[r]^-{\cong} & \mathbb{V}{\rm ert}(\sVir_{c}) }
\een
where ${\bf Vir}_c$ denotes the Virasoro vertex algebra of charge
$c$. 
\end{theorem}

After spelling out the local structure, we study global
sections, or the factorization homology, of $\sVir_{c}$. Some care must be taken when defining the cocycle determining
the extension on the Dolbeault resolution of vector fields on a
general Riemann surface since the original cocycle for the local Lie
algebra on $\bfC$ is coordinate dependent. We show that a slightly
modified version of the cocycle gives a coordinate independent
description and hence a {\it universal} version of the cocycle. That
is, we show that the Virasoro factorization algebra defines a
factorization algebra on the site of Riemann surfaces. We calculate
the cohomology of global sections of the Virasoro factorization
algebra and write down correlation functions. 

This paper can be viewed in conjunction with a new direction of work
that combines methods of renormalization, homological perturbation
theory, and factorization algebras developed in Costello \cite{C1}
and Costello-Gwilliam \cite{CG1, CG2}. From the data of a classical
field theory, defined in terms of an action functional, one applies of
homotopical renormalization to construct a quantization. Locality of
the theory and the quantization on the manifold in which the theory
lives combine to give the structure of a {\em factorization algebra}
on the algebraic observables of the theory. 

The last section of this paper exhibits how the usual physical idea
of the Virasoro algebra encoding the symmetries of a conformal field
theory fits in to the model for QFT developed by
Costello-Gwilliam. The usual Virasoro
symmetry in field theory is naturally encoded by map of factorization algebras
from the Virasoro factorization algebra (at a certain central charge)
to the factorization algebra of observables. We will focus on a
particular example of a chiral conformal field theory, called the {\em
  free $\beta\gamma$ system}, though the methods we use work in a much
larger context. 

\subsection{Notation and conventions}
\begin{itemize}

\item If $X$ is a complex manifold we have a decomposition of the tangent bundle $T^{1,0} X \oplus T^{0,1}X$. Unless otherwise noted we
will write $TX = T^{1,0} X$ for the $(1,0)$ part of the tangent
bundle. With respect to this decomposition the de Rham differential 
\ben
\dd_{\rm dR} : \scr{O}(X) \to \Omega^1(X) = \Omega^{1,0}(X) \oplus
\Omega^{0,1}(X) := \Gamma\left((T^{1,0}X)^\vee\right) \oplus
\Gamma\left((T^{0,1}X)^\vee\right)
\een
splits as $\partial + \dbar$. 

\item Let $V$ be a graded vector space. We denote by ${\rm Tens}(V)$ the full
tensor algebra of $V$, $\tensor_{n \geq 0} V$. This is again a graded
vector space in the natural way. Define the symmetric
algebra as 
\ben
{\rm Sym}(V) = \bigoplus_{n \geq 0} {\rm Sym}^n(V)
\een
where ${\rm Sym}^n(V) = \left(V \tensor \cdots \tensor
  V\right)_{\Sigma_n}$. We will also need the completed symmetric
algebra
\ben 
\Hat{\rm Sym}(V) = \prod_{n \geq 0} {\rm Sym}^n(V) . 
\een

\item All graded vector spaces are cohomologically graded. For $k \in \bfZ$ we denote by $V[k]$ the graded vector space with
graded components:
\ben
(V[k])^i = V^{i+k} .
\een
If $W$ is an ordinary (ungraded) vector space, we will understand it
as a graded vector space concentrated in degree zero. For instance,
$W[k]$ is concentrated in degree $-k$. 

\item Let $\mathfrak{g}$ be a dg Lie algebra. That is, a
  $\mathbf{Z}$-graded vector space together with a differential
  $\dd_{\mathfrak{g}} : \mathfrak{g}^\bullet \to \mathfrak{g}^{\bullet + 1}$ of degree +1 and a
  bracket $[-,-]$ that is graded antisymmetric, satisfies the graded
  Jacobi identity, and for which $\dd_{\mathfrak{g}}$ is a graded derivation. We
  define Chevalley-Eilenberg chains for computing Lie algebra homology as
\ben
C_*(\mathfrak{g}) \; := \; {\rm Sym} \left(\mathfrak{g}[1]\right) =
\bigoplus_{n \geq 0} {\rm Sym}^n(\mathfrak{g}[1])
\een
with differential given by $\dd = \dd_{\mathfrak{g}} + \dd_{\rm
  CE}$ where $\dd_{\rm CE}$ is the usual CE-differential determined by
$\dd_{\rm CE} (a \wedge b) = [a, b]$ on ${\rm Sym}^2$. Similarly, Chevalley-Eilenberg cochains for computing Lie
algebra cohomology are defined by
\ben
C^*(\mathfrak{g}) \; := \; \Hat{\rm Sym}
\left(\mathfrak{g}^\vee[-1]\right)
\een
with differential given by $\dd = \dd_{\mathfrak{g}}^\vee + \dd_{\rm
  CE}^\vee$. 

\item We will need to consider a topology on the dg vector spaces we
  work with. Unless otherwise noted our complexes will take values in the
  category of dg nuclear vector spaces. This is an especially
  convenient category of topological vector spaces which are locally
  convex and Hausdorff. For more on their properties see \cite{C1} Given two dg nuclear vector
  spaces we denote by $V \tensor W$ the completed tensor product. This
  tensor product
  makes the category of dg nuclear vector spaces a symmetric monoidal
  category which we denote by ${\rm dgNuc}^{\tensor}$. 

\end{itemize}

\subsection{Acknowledgements}
I'd like to thank Owen Gwilliam and Ryan Grady for their shared interest in this
project and generosity in providing detailed comments and
discussion pertaining to this work. I have also benifited from useful
comments from and discussions with Chris Elliott, Ben
Knudsen, and Philsang Yoo. Finally, I'd like to thank Kevin Costello for
sparking my interest in this project.

\section{Virasoro as a local Lie algebra on $\bfC$}

In this section we introduce a local version of the Virasoro Lie
algebra on the complex plane. It appears as an extension of the Lie
algebra of holomorphic vector fields on $\bfC$ given by an explicit
cocycle. 

\subsection{Dolbeault resolution of holomorphic vector fields}

Let $X$ be a complex manifold. We study the space of
holomorphic sections of the holomorphic $(1,0)$ tangent bundle
$\scr{O}^{hol}(TX)$. We can use
the decomposition of the tangent bundle above gives us a resolution
for this space. Indeed, the $\dbar$ operator extends to define a complex
\ben
\xym{
\Omega^{0,0}(X,TX) \ar[r]^-{\dbar} & \Omega^{0,1}(X,TX)
\ar[r]^-{\dbar} & \Omega^{0,2}(X,TX) \ar[r]^-{\dbar} & \cdots .
}
\een

We will be concerned with the case that the complex manifold is a
Riemann surface $\Sigma$. Indeed the Dolbeaut complex above defines the
dg Lie algebra $\scr{L}^\Sigma := \Omega^{0,*}(\Sigma,T\Sigma)$. The differential is $\dbar$ and
the Lie bracket is given by extending the ordinary Lie bracket on
$\Omega^{0,0}(\Sigma,T\Sigma)$ to a graded Lie bracket. 

We will consider $\scr{L}^\Sigma$ as a sheaf of cochain
complexes that assigns to an open $U
\subset \Sigma$ the complex
\ben
(\Omega^{0,*}(U,TU) , \Bar{\partial}) .
\een
Moreover, $\scr{L}^\Sigma$ is a sheaf
of dg Lie algebras. In fact it has even more structure, that of a
{\it local dg Lie algebra}. 

The following definition can be found in \cite{CG1}.

\begin{defn}
A {\it local dg Lie algebra} on a manifold $M$ is the following data:
\begin{itemize}
\item[(1)] A graded vector bundle $L$ on $M$, whose sheaf of smooth
  sections is denoted $\scr{L}$.
\item[(2)] A differential operator $\dd : \scr{L} \to \scr{L}$ of
  degree one and square $0$.
\item[(3)] Antisymmetric multi-differential operators
\ben
\dd : \scr{L} \to \scr{L} \;\; , \;\; [-,-] : \scr{L}^{\tensor 2} \to \scr{L}
\een
of degree one and zero respectively, that give $\scr{L}$ the structure of a sheaf of $L_\infty$-algebras.
\end{itemize}
\end{defn}

Since $\dbar$ and the Lie bracket of vector fields are differential and
bi-differential operators, respectively we see that $\scr{L}^\Sigma$ is
a local $L_\infty$-algebra. 

We will also be interested in the compactly supported version of
$\scr{L}^\Sigma$; it assigns to an open the dg vector space
\ben
(\Omega_c^{0,*}(U,TU), \Bar{\partial})
\een
Analogously, this is a precosheaf of dg Lie algebras which we denote
by $\scr{L}^\Sigma_c$. 

\subsection{Lie algebra extensions and the cocycle}

We are interested in a one-dimensional central extension of
$\scr{L}^\Sigma$. As the Lie algebra in question is local, we ask
for our extensions to be local as well. Before defining what we mean
by this, we review extensions of ordinary Lie and dg Lie algebras. 

In the remainder of this section, as well Sections \ref{secann} and
\ref{secvert} we will be concerned with the case that the Riemann
surface is the complex line $\Sigma = \bfC$.

\subsubsection{Extensions}
A central extension $\Hat{\mathfrak{g}}$ of an ordinary Lie algebra
$\mathfrak{g}$ is a Lie algebra that fits into an exact sequence
\ben
0 \to \bfC \to \Hat{\mathfrak{g}} \to \mathfrak{g}\to 0
\een
such that $[\lambda, x] = 0$ for all $\lambda \in \bfC$ and $x \in
\mathfrak{g}$. Isomorphism classes of central extensions of
$\mathfrak{g}$ are in bijective correspondence with ${\rm H}^2(\mathfrak{g})$. 

For a dg Lie algebra $\mathfrak{g}$ and an integer $k$ we can define a
$k$-shifted central
extension of $\mathfrak{g}$. It fits into and exact
sequence
\be\label{ext1}
0 \to \bfC[k] \to \Hat{\mathfrak{g}} \to \mathfrak{g}\to 0
\ee
and satisfies $[\lambda,x] = 0$ as above. 

\begin{remark} The group ${\rm
  H}^{2+k}(\mathfrak{g})$ does {\it not} parametrize such
extensions. It parametrizes a larger class of extensions, namely
shifted $L_\infty$-extensions of $\mathfrak{g}$. That is, exact
sequences as in (\ref{ext1}) except $\Hat{\mathfrak{g}}$ is allowed to
be an $L_\infty$-algebra, and the maps are of $L_\infty$-algebras. 
\end{remark}

\begin{example} Consider the Lie algebra of vector fields on $S^1$,
  ${\rm Vect}(S^1)$. This is an ordinary Lie
  algebra that, as usual, can be thought of as a dg Lie algebra concetrated in degree zero. The
  Gelfand-Fuks extension mentioned in the introduction is a 2-cocycle, hence determines an element
  in ${\rm H}^2_{\rm Lie}({\rm Vect}(S^1))$. In fact, this cohomology is one
  dimensional. See \cite{GF} for a proof of this. 
\end{example}

Now, let $\scr{L}$ be a local dg Lie algebra on a manifold $M$. A {\it local}
$k$-{\it shifted central extension} of $\scr{L}$ is a dg Lie algebra
structure on the precosheaf
\ben
\Hat{\scr{L}}_c = \scr{L}_c \oplus \ul{\bfC}[k]
\een
such that for all opens $U \subset M$:
\begin{itemize}
\item (Central) For any $\lambda \in \bfC[k]$ and $x \in \Hat{\scr{L}}_c(U)$ we
  have $[x,\lambda]=0$ and the sequence
\ben
0 \to \bfC[k] \to \Hat{\scr{L}}_c(U) \to \scr{L}_c(U) \to 0
\een
is exact. 
\item (Local) The differential $\dd_{\Hat{\scr{L}}} : \scr{L}_c(U) \to \bfC[k]$ and Lie
  bracket $[-,-] : \scr{L}_c(U) \tensor \scr{L}_c(U) \to \bfC[k]$ both
  factor through the $k$-shifted integration map
\ben
\int_U : {\rm Dens}_c(U) [-k] \to \bfC[k] .
\een
Here, ${\rm Dens}_c$ denotes the cosheaf of compactly supported densities. 
\end{itemize}

\begin{remark} As in the ordinary case, there is a cohomology that parametrizes
central extensions of this local nature. Let $\scr{L}$ be a local Lie
algebra on $M$. In \cite{C1} local functionals are defined as 
\ben
C^*_{\rm loc,red}(\scr{L}) := {\rm Dens}_M \tensor_{\scr{D}_M} C^*_{\rm
  red}\left({\rm Jet}(L)\right) .
\een
Here, $\scr{D}_M$ is the space of differential operators on $M$ and
${\rm Jet}(L)$ is the infinite Jet-bundle of the vector bundle
$L$. The jet-bundle inherits a natural $\scr{D}_M$-module structure,
and this induces one on cochains. There is another interpretation of
local cochains. They are precisely the graded multilinear functionals
on $\scr{L}$ that factor through the integration map. More precisely,
integration along $M$ induces a natural inclusion
\ben
C^*_{\rm loc,red}(\scr{L}) \hookrightarrow C^*_{\rm red}(\scr{L}_c(M))
\een 
that sends a local functional $S$ to the functional $\varphi \mapsto
\int_M S(\varphi)$. 

Just as in the case of (non-local) dg Lie algebras, the degree $2+k$
cocycles of $C^*_{\rm loc,red}(\scr{L})$ parametrize a larger class of
extensions, namely local $L_\infty$-algebra extensions of
$\scr{L}$. For our situation, when $\scr{L}$ is the Dolbeault
resolution of holomorphic vector fields on $U \subset \bfC$ (or on a
closed Riemann surface $\Sigma$) the non-trivial degree one cocycles
are all cohomologous to one of the form
\ben
\scr{L}_c(U)^{\tensor 2} \to \bfC
\een
and hence all (-1)-shifted extensions will be equivalent to a local dg Lie
algebra. 
\end{remark}

\subsubsection{A cocycle for $\scr{L}^\bfC$} \label{sec shiftext}

We now define the cocycle used to construct the central extension of
$\scr{L}^\bfC$ we are interested in. Let $U
\subset \bfC$ and fix a coordinate. Consider the bilinear map
\ben
\omega : \scr{L}^\bfC_c(U) \tensor \scr{L}^\bfC_c(U) \to \bfC
\een
given by
\ben
(\alpha \tensor \partial_z, \beta \tensor \partial_z) \mapsto \frac{1}{2\pi}\frac{1}{12}\int_U \left(
  \partial_z^3 \alpha_0 \beta_1 +
  \partial_z^3\alpha_1 \beta_0
\right) \d^2 z 
\een
where $\alpha = \alpha_0+\alpha_1\dd \Bar{z}$ and $\beta = \beta_0
  +\beta_1 \dd \Bar{z}$. One checks by direct calculation that this is
  a cocycle and is our analog of the Gelfand-Fuchs cocycle.

This cocycle defines for us a local $(-1)$-shifted central extension
$\Hat{\scr{L}}^\bfC$ of $\scr{L}^\bfC$ via the local extension construction above. As a
cosheaf of vector spaces it is
\ben
\scr{L}^\bfC_c \oplus \bfC \cdot c [-1] .
\een
On an open $U$, the Lie bracket is defined by the rules
\ben
[\alpha \tensor \partial_z, \beta \tensor \partial_z]_{\Hat{\scr{L}}^\bfC_c} \; := \;
[\alpha \tensor \partial_z, \beta \tensor \partial_z] + \frac{1}{2\pi}\frac{1}{12}\int_U \left(
  \partial_z^3 \alpha_0 \beta_1 +
  (\partial_z^3\alpha_1 \beta_0
\right) \d^2 z  \cdot c
\een
and $[\alpha \tensor \partial_z, c]_{\Hat{\scr{L}}^\bfC_c} = 0$. 

The locality and cocycle properties imply that $\omega$ determines an
  element in ${\rm H}^1_{\rm loc}(\scr{L}^\bfC)$. 

\subsection{Statements about cohomology}
The following facts about the $\dbar$-cohomology of subsets of $\bfC$
will be used throughout. Let $U \subset \bfC$ be open. The following lemma is due to Serre \cite{Serre}. 

\begin{lemma} The compactly supported Dolbeault cohomology of $U$ is
  concentrated in degree 1 and there is a continuous isomorphism
\ben
{\rm H}^1(\Omega_c^{0,*}(U),\Bar{\partial}) \; \cong \;
\left(\Omega^1_{\rm hol}(U) \right)^\vee .
\een
\end{lemma}
Here, $(-)^\vee$ denotes the continuous linear dual of nuclear
Fr\'{e}chet spaces. Explicitly, we assign to a $(0,1)$-form $\alpha$ on $U$, the continuous
linear functional 
\ben
\<\alpha,-\> : \Omega^1_{\rm hol} (U) \to \bfC \;\; , \;\; \beta \mapsto \int_U
\alpha \beta .
\een

Next, we need the following fact about dg Lie algebras. 

\begin{lemma} Suppose $L$ is a dg Lie algebra such that ${\rm H}^*(L)$
  is concentrated in a single degree. Then $L$ is formal (as a dg Lie
  algebra). 
\end{lemma}
\begin{proof} 
Suppose the cohomology of $L$ is concentrated in degree $m$. Define
the subcomplex $L' \hookrightarrow L$ as follows: for $k < m$ set
$(L')^k := L^k$, for $k = 0$ set $(L')^0 = \ker(d_L : L^0 \to L^1)$, for
$k > m$ set $(L')^k := 0$. There is a natural zig-zag of dgla's
\ben
L \hookleftarrow L' \to {\rm H}^0L .
\een
Both arrows are clearly weak equivalences. 
\end{proof}

Serre's result implies that $\scr{L}^\bfC_c(U) = \Omega^{0,*}_c(U,TU)$ is formal
for all opens $U \subset \bfC$. In fact, there is a quasi-isomorphism
of dg Lie algebras 
\ben
\Omega^{0,*}_c(U,TU) \; \simeq \; {\rm H}(\Omega^{0,*}_c(U,TU), \dbar)
\; \cong \; \left(\Omega_{\rm hol}^1(U,TU) \right)^\vee .
\een

This implies the following useful fact about the Lie algebra
cohomology.

\begin{prop} Let $U \subset \bfC$ be open. Then,
\ben
{\rm H}^{\rm Lie}_*(\scr{L}_c^\bfC(U)) := {\rm H}^*\left({\rm Sym}(\scr{L}_c^{\bfC} (U)[1]), \dbar + \dd_{CE} \right) \; \cong \;
{\rm Sym}\left(\Omega^1_{\rm hol}(U,TU)^\vee\right)
\een
concentrated in degree $0$.
\end{prop}
Here, we extend the differential $\dbar$ on $\scr{L}_c(U)$ to the
symmetric algebra in the obvious way. 
\begin{proof} 
This result follows from formality. Indeed,
\ben
{\rm H}_*^{\rm Lie}(\scr{L}_c^\bfC(U)) \; \cong \; {\rm H}_*^{\rm Lie}({\rm
  H}_{\dbar}^*(\scr{L}_c^\bfC(U)) \; = \; {\rm H}_*^{\rm Lie}(\Omega^1_{\rm
  hol}(U,TU)^\vee) .
\een
Now, $\Omega_{\rm hol}^1(U,TU)^\vee$ is an abelian dg Lie algebra
concentrated in a single degree. Thus 
\ben
{\rm H}^*_{\rm Lie}(\Omega^1_{\rm
  hol}(U,TU)^\vee) = {\rm Sym}\left(\Omega^1_{\rm
    hol}(U,TU)^\vee\right)
\een
as desired. 
\end{proof}

This result also follows from considering the filtration
spectral associated to symmetric tensor power degree. The $E_1$-page is ${\rm
  Sym}\left({\rm H}_{\dbar}^*(\scr{L}_c^\bfC(U))\right)$ and the
spectral sequence degenerates
at the $E_2$-page as the $\dbar$-cohomology is concentrated in a single
degree. In fact, the degeneration of this spectral sequence
  associated to cochains on a dg Lie algebra $\mathfrak{g}$ is closely related to the
  formality of $\mathfrak{g}$, for example see \cite{man}.

\begin{remark} In the second part of the paper we consider closed
  Riemann surfaces. It is still true that on a closed Riemann surface,
  the spectral sequence associated to the dg Lie algebra
  $\Omega^{0,*}(\Sigma,T\Sigma)$ degenerates. In fact, this dg Lie
  algebra is also formal. 
\end{remark}

\subsection{Factorization algebras}
Central to this work is the notion of a {\it factorization
  algebra}. We recall the relavent theory as in
\cite{CG1}. 

Fix a topological space $M$. For the level of generality of most of this section we
work in an arbitrary symmetric monoidal category $\scr{C}^\tensor$
closed under small colimits. For
the purposes of this work we are mainly concerned with $\scr{C} = {\rm
  dgNuc}$, the dg category of cochain complexes of nuclear vector spaces over
$\bfC$ with symmetric monoidal structure given by the completed tensor
product over $\bfC$. 

\subsubsection{Prefactorization} A {\it prefactorization algebra} $\sF$ on
$M$ with values in $\scr{C}^\tensor$ is an assignment of an object
$\sF(U)$ of
$\scr{C}$ for each open $U \subset M$
together with the following data:
\begin{itemize}
\item For $U \subset V$, a morphism $\sF(U) \to \sF(V)$. 
\item For any finite collection $\{U_i\}$ of pairwise disjoint opens
  in an open $V \subset M$ a morphism
\ben
\tensor_i \sF(U_i) \to \sF(V) .
\een
\item Coherences between the above two sets of data.
\end{itemize}

For a better definition we need to define the following symmetric monoidal category ${\rm
  Fact}(M)^{\sqcup}$. Its objects are topological spaces $U$ together with a
map $U \to M$ such that on each connected component of $U$ this map is
an open embedding. A morphism from $U \to M$ to $V \to M$ is a
commutative diagram
\ben
\xym{
U \ar[rr]^i \ar[dr] & & V \ar[dl] \\ 
& M & 
}
\een 
with $i$ an open embedding. Composition is done in the obvious way. The symmetric monoidal structure is given
by disjoint union. 

A more
precise definition of a {\it prefactorization algebra} is symmetric
monoidal functor 
\ben
\sF : {\rm Fact}(M)^{\sqcup} \to \scr{C}^\tensor .
\een

\begin{example} The coherence of the data above can be read of
  immediately from this definition and encodes the transitivity of
  opens. For instance, suppose $U_1, U_2 \subset V \subset W$ are
  opens with $U_i$ disjoint. Then $\sF$ applied to this
  composition says that 
\ben
\xym{
\sF (U_1) \tensor \sF(U_2) \ar[r] \ar[dr] & \sF(V) \ar[d]
\\ & \sF(W)
}
\een
commutes. 
\end{example}

The structures we consider in the first part of this paper are
completely encoded by a {\em pre}factorization structure. In the last section, however;
when we will be concerned with global sections on a general Riemann
surface, it is critical that our object satisfies a form of descent. 

\subsubsection{Factorization: gluing} A {\it factorization algebra} is
a prefactorization algebra satisfying a descent axiom. Descent for
ordinary sheaves (or cosheaves) says that one can recover the value of
the sheaf on large open sets by breaking it up into smaller
opens. That is, if $\scr{U} = \{U_i\}$ is a cover of $U \subset M$
then a presheaf $\sF$ of vector spaces is a sheaf iff 
\ben\xym{
\sF(U) \to \bigoplus_i \sF(U_i)
\ar@<0.5ex>[r]\ar@<-0.5ex>[r] & \oplus_{i,j} \sF(U_i \cap U_j) 
}
\een
is an equalizer diagram for all opens $U$ and covers $\scr{U}$. It is convenient to introduce the \v{C}ech
complex associated to $\scr{U}$. The $p$th space is
\ben
\check{C}^p(\scr{U},\sF) := \bigoplus_{i_0,\ldots,i_p}
\sF(U_{i_0} \cap \cdots \cap U_{i_p}) .
\een
The differential $\check{C}^p \to \check{C}^{p+1}$ is induced from the
natural inclusion maps $U_{i_0} \cap \cdots \cap U_{i_p}
\hookrightarrow U_{i_0} \cap \cdots \cap \Hat{U}_{i_j} \cap \cdots
\cap U_{i_p}$. The sheaf condition is equivalent to saying that the
natural map
\ben
\sF(U) \to {\rm H}^0(\check{C}(\scr{U},\sF)) 
\een
is an isomorphism. There is a similar construction for cosheaves, but
the arrow goes in the opposite direction. 

We are interested in descent for a different topology, that is, for
only a special class of open covers. Call an open cover $\scr{U} = \{U_i\}$ of
$U \subset M$ a {\it Weiss cover} if for any finite collection of points
$\{x_1,\ldots,x_k\}$ in $U$, there exists an open set $U_i$ such that
$\{x_1,\ldots,x_k\} \subset U_i$. This is equivalent to providing a
topology on the Ran space. 

A Weiss cover defines a Grothendieck topology on ${\rm
    Op}(M)$, the poset of opens in $M$. A {\it factorization algebra}
  on $M$
  is a prefactorization algebra on $M$ that is, in addition, a
  homotopy cosheaf for this Weiss topology. 

When $\scr{C}^\tensor = {\rm dgVect}$ we can be explicit about this
homotopy gluing condition using a variant of the \v{C}ech complex
above. Let $\sF$ be a cosheaf of dg vector spaces. For $\scr{U} = \{U_i\}_{i \in I}$ let
$\check{C}^p(\scr{U}, \sF)$ be the complex
\ben
\bigoplus_{i_0,\ldots,i_p} \sF(U_{i_1} \cap \cdots \cap U_{i_k})
[p-1] 
\een
with differential inherite from $\sF$. Then
$\check{C}(\scr{U},\sF)$ is a bigraded object. The differential is the total differential obtained from
combining the ordinary
\v{C}ech differentials above plus the internal differential of
$\sF$. The cosheaf
condition is that the natural map
\ben
\check{C}(\scr{U},\sF) \to \sF(U)
\een
is an equivalence for all Weiss covers $\scr{U}$ of $U$. 

\begin{remark} One might refer to this as a {\it homotopy}
  factorization algebra, reserving a {\it strict} factorization algebra for
  one in which 
\ben
\check{\rm H}^0(\scr{U},\sF) \to \sF(U)
\een
is an equivalence. The $\check{\rm H}^0$ means we have only taken
cohomology with respect to the \v{C}ech differentials. It has a
natural dg structure inherited from $\sF$. 
\end{remark} 

\subsection{(Twisted) Envelopes}
One of the most useful ways of constructing factorization algebras is
the ``factorization envelope'' of a local Lie algebra. This is the
analog of the unverisal enveloping algebra of a Lie algebra.

Let $\scr{L}$ be any local Lie algebra on a manifold $M$. Denote
by $\scr{L}_c$ its associated cosheaf of compactly supported
sections. Define
the prefactorization algebra $U^{\rm fact} \scr{L}$ as follows:
\begin{itemize}
\item For an open $U \subset M$ we assign the complex
  $C_*(\scr{L}_c(U))$ with it's usual differential
  $\dd = \dd_{\scr{L}} + \dd_{\rm CE}$. 
\item Suppose $\sqcup_i U_i \hookrightarrow V$ is an inclusion of
  disjoint opens inside a bigger open. The structure maps of the
  prefactorization algebra come from applying $C_*(-)$ to the
  structure maps of the cosheaf 
\ben 
\oplus_i
  \scr{L}_c(U_i) \to \scr{L}_c(V).
\een
\end{itemize}
In fact, we will use the following fact to compute global sections,
i.e. factorization homology. 

\begin{theorem}[\cite{CG1}] The prefactorization algebra $U^{\rm
    fact}\scr{L}$ satisfies descent, that is it is a
  factorization algebra. 
\end{theorem}


\begin{example} If $\mathfrak{g}$ is an {\it ordinary} Lie algebra we
  can consider the local Lie algebra $\Omega^*_{\bfR} \tensor
  \mathfrak{g}$ on $\bfR$. The factorization algebra $U^{\rm fact}
  (\Omega_\bfR \tensor \mathfrak{g})$ is locally constant on
  $\bfR$. Now, map that sends a factorization algebra on $\bfR$ to
  its value on an interval is known to induce an equivalence of categories
\ben
\{A_\infty{\rm -algebras}\} \simeq \{E_1{\rm -algebras}\} \simeq \{{\rm
  locally\;constant\;factorization\;algebras\;on\;} \bfR\} .
\een
Under this equivalence the factorization algebra $U^{\rm fact}
\mathfrak{g}$ corresponds to the associative algebra $U \mathfrak{g}$. 
\end{example}

Now, suppose we have an element $\omega \in {\rm H}_{\rm loc}^1(\scr{L})$ corresponding to a
$(-1)$-shifted central extension $\Hat{\scr{L}}$ of a local Lie
algebra $\scr{L}$ on a manifold $M$. We define the {\em
  twisted factorization envelope} $U_\omega^{\rm fact} \scr{L}$ as
the factorization algebra on $M$ that sends an open $U \subset M$ to the complex
\ben
\left({\rm Sym}\left(\scr{L}(U)[1] \oplus
    \bfC \cdot C \right), \dd_{\scr{L}} + \dd_{\rm CE} + \omega\right)
\een
where $\omega$ is made into an operator on ${\rm Sym}$ as
follows. On ${\rm Sym}^{\leq 1}$ it is zero and on ${\rm Sym}^2$ it is
\ben
(\alpha, \beta) \mapsto
C \cdot \omega(\alpha, \beta) .
\een 
It is extended to the full symmetric algebra by demanding that it is a
graded derivation. Note that $U_\omega^{\rm fact} \scr{L}^\bfC =
U^{\rm fact} \Hat{\scr{L}}^\bfC_c$ so that the twisted envelope is
just the envelope of the extended local Lie algebra. 

\subsubsection{The Virasoro factorization algebra}

We will now specialize to factorization algebras valued in the
symmetric monoidal
category of dg nuclear vector spaces $\dgNuc$ or slight variants
thereof. 

In the remainder of the paper we are interested in both the untwisted and twisted 
factorization envelopes of the local Lie algebra of holomorphic vector
fields.

First, define the {\em Virasoro factorization algebra at central
  charge zero} by
\ben
\sVir_0 := U^{\rm fact} \scr{L}^\bfC .
\een 
This is a factorization algebra valued in the category $\dgNuc$ (since the
Dolbeault complexes belong to this category). 

Let $\omega \in {\rm H}^1_{\rm loc}(\sL^\bfC)$ denote the cocycle from
Section \ref{sec shiftext}. We define the {\em Virasoro factorization
  algebra} by
\ben
\sVir := U^{\rm fact}_{\omega} \scr{L}^\bfC  .
\een
The factorization algebra $\sVir$ is a factorization algebra in
the category of $\bfC[c]$-modules in dg nuclear vector spaces. In
particular, we can specialize a value of $c$ to obtain a factorization
algebra in $\dgNuc$. We will denote such a specialization by $\sVir_c$
and call it the {\em Virasoro factorization algebra of central charge
  $c$}. 

\section{Annuli: recovering the Virasoro} \label{secann} 
In this section we show how the Virasoro Lie algebra is encoded in the
factorization algebras constructed above. 

First we recall the definition the Virasoro
Lie algebra. Consider the ring of Laurent power series in one
variable $\bfC ((t))$. As a vector space the Lie algebra of
derivations ${\rm W}_1^\times := {\rm Der}\left(\bfC ((t)) \right)$ is isomorphic to $\bfC
((t)) \partial_t$. The ring $\bfC ((t))$ is equal to functions on the
holomorphic formal
punctured disk $\Hat{D}^\times$ and ${\rm W}_1^\times$ is the Lie
algebra of formal vector fields on
the punctured disk. Let ${\rm Vir}$ be
the central extension of ${\rm W}_1^\times$ determined by the
Gelfand-Fuks cocycle $\omega_{\rm GF}$ defined in the introduction. It fits into the exact sequence of Lie algebras
\ben
0 \to \bfC \cdot C \to {\rm Vir} \to {\rm W}_1^\times \to 0 .
\een
Thus, as a vector space we have ${\rm Vir} = \bfC((t)) \partial_t
\oplus \bfC \cdot C$. 
Explicitly, the bracket in this Lie algebra is
\ben
[f(t) \partial_t, g(t) \partial_t] = (f(t) g'(t) - f'(t)
g(t))\partial_t +
\frac{1}{12} \oint f'''(t) g(t) \d t \cdot C .
\een
It is topologically
generated by $c$ and $L_n = t^{n+1} \partial_t$ and in terms of these
generators, the commutator is
\ben
[L_n,L_m] = (n-m) L_{n+m} + \frac{m^3-m}{12} \delta_{n,-m} \cdot C .
\een 

Now, consider the universal enveloping algebra of the Virasoro Lie
algebra $U({\rm Vir})$. Being an associative
algebra it determines a locally constant factoriztation algebra on
$\bfR_{>0}$. Denote this factorization algebra by $\scr{A}_{\rm
  Vir}$. Explicitly, $\scr{A}_{\rm Vir}$ sends an interval $I$ to
$U({\rm Vir})$ (considered as a dg vector space concentrated in degree
zero) and the structure maps are induced by the usual
associative multiplication on $U({\rm Vir})$. 

Let $\rho : \bfC^\times \to
\bfR_{>0}$ be the map $z \mapsto z \Bar{z}$. We consider the
push-forward factorization algebra $\rho_* \sVir$. This is a
factorization algebra on $\bfR_{>0}$. The main result of this section can be stated as follows. 
\begin{prop}\label{phi} There is a map of factorization algebras 
\be\label{dense}
\Phi : \scr{A}_{\rm Vir} \to {\rm H}^0(\rho_* \sVir)
\ee
that is a dense inclusion of topological vector spaces on evey open
interval $I \subset \bfR_{>0}$. 
\end{prop}

Note that on an open
interval $I \subset \bfR_{>0}$
\ben
(\rho_*\sF_\omega)(I) \; = \; \sVir (\rho^{-1}(I)) .
\een
So, we need to understand what $\sVir$ does to annuli. 

\begin{remark} This proposition says that every cohomology class in
  $\sVir$ applied to
  an annulus is arbitrarily close to some element of the universal
  enveloping algebra of the Virasoro Lie algebra. Moreover, the
  structure maps of the factorization algebra are the continuous
  extensions of the multiplication for $U{\rm Vir}$. 
\end{remark}

\subsection{The case of zero central charge}

Recall that we have the following identification for any open $U \subset \bfC$:
\ben
{\rm H}^*(\sVir_0(U)) \; \cong \; {\rm Sym} \left({\rm
    H}^1(\Omega^{0,*}_c(U,TU)) \right) \; \cong \; {\rm Sym} \left( \Omega^1_{\rm
    hol}(U,TU)^\vee \right)
\een
concentrated in cohomological degree $0$. 

First, we describe the untwisted version of the map (\ref{dense}),
denote it $\Phi_{un} : \bfC((z)) \partial_z \to \rho_*(\sVir_0)$. Let $L_n = z^{n+1} \partial \in \bfC((z)) \partial_z$ be the usual basis
vectors for $n \in \bfZ$. Pick an open interval $I \subset
\bfR_{>0}$ and let $A = \rho^{-1}(I)$. We will utilize a function $f :
\bfC^\times \to \bfR$ for $A$ that satisfies the following:
\begin{itemize}
\item $f$ is only a function of $r^2 = z \Bar{z}$. 
\item $\int_{A} f \; \dd z \dd\Bar{z} = 1$.  
\item $f(r^2) \geq 0$ and $f$ is supported on $A$. 
\end{itemize}
We will refer to $f$ as a {\it bump function} for $A$. Finally, we
define
\ben
\Phi_{un}(I) : L_n\mapsto \left \lfloor f(z\zbar) z^{n+2} \dd
\zbar \partial_z \right\rfloor
\een
where $\left\lfloor - \right\rfloor$ denotes the cohomology class in compactly
supported Dolbeault forms. Note that this map is a dense inclusion of topological vector spaces
by Serre's resulted stated above. Therefore, we might unambiguously
confuse $L_n$ with its image in ${\rm H}^*(\sVir_0 (A))$. Also, it will be convenient to use the notation $L_n(A) = f(z \zbar) z^{n+2}
\dd \zbar \partial_z$ for the lift of $L_n$ to the factorization
algebra. We make no reference to the bump function chosen since this choice will not affect the cohomology class. 

Consider three nested disjoint annuli $A_1,A_2,A_3$ where $A_i$ has
inner radius $r_i$ and outer radius $R_i$ so that $R_1 < r_2$ and $R_2
< r_3$. Suppose all three are contained in the big annuli $A$,
i.e. $r < r_1$ and $R_3 < R$. 

Let's explain some notation for the factorization product of such
nested annuli. The relavent factorization maps are
\bestar
\bullet & : & \sVir_0 (A_2) \tensor \sVir_0 (A_1) \to \sVir_0 (A) \\
\bullet & : & \sVir_0(A_3) \tensor \sVir_0(A_2) \to \sVir_0(A) .
\eestar
Moving outward, radially, corresponds to multiplying from the right to
left in this notation. This is known as {\it radial ordering}. Using
this notation, upon taking cohomology we want to show
\ben
L_m \bullet L_n- L_n \bullet L_m = (m-n) L_{n+m} .
\een

\begin{remark} This is a bit of abuse of notation, as we are using the
  same symbol $L_m$ even though the two live in different
  spaces. This is a superficial confusion since $\Phi_{un}$ is an
  embedding, but what the above
  expression actually means is
\ben
\Phi_{un}(\rho(A_2))(L_m) \bullet \Phi_{un}(\rho(A_1))(L_n) -
\Phi_{un}(\rho(A_3))(L_n) \bullet \Phi_{un}(\rho(A_2))(L_m) = (m-n)
\Phi_{un}(\rho(A))(L_{n+m}) .
\een 
\end{remark}

Let $f_i: \bfC^\times \to \bfR$ be a bump function for $A_i$,
$i=1,2,3$. We use these to obtain lifts of $L_n$'s to the
factorization algebra. Explicitly, $L_m(A_1) \in \sVir_0(A_1)$,
$L_m(A_3) \in \sVir_0(A_3)$ and $L_n(A_2) \in \sVir_0(A_2)$. 

Now, in cohomology
\ben
\left\lfloor L_m(A_1) L_n(A_2) - L_n(A_2) L_m(A_3) \right\rfloor = L_m \bullet L_n- L_n \bullet
L_m
\een
and
\ben
(m-n) L _{m+n} = \left\lfloor [L_m,L_n] (A) \right\rfloor = \left\lfloor f_2(r^2) (m-n)
z^{n+m+2} \dd \Bar{z} \tensor \partial_z \right\rfloor  .
\een

Consider the function
\ben
F(z,\Bar{z}) = z^{m+1} \int_{0}^{z\Bar{z}} f_1(s) - f_3(s) \; \dd s .
\een
We compute the $\dbar$ operator acting on $F(z,\zbar)$ as
\bestar
\Bar{\partial} (F(z,\Bar{z})) & = & z^{m+1} \frac{\partial}{\partial
  \Bar{z}} \left(\int_{0}^{z\Bar{z}} f_1(s) - f_3(s) \; \dd s \right)
\dd \Bar{z}
\\ & = & z^{m+1} \frac{\partial(z \Bar{z})}{\partial \Bar{z}}
\frac{\partial}{\partial(z \Bar{z})} \left(\int_{0}^{z\Bar{z}} f_1(s)
  - f_3(s) \; \dd s \right) \dd \Bar{z} \\ & = & z^{m+2} \left(f_1(z
\Bar{z}) - f_3(z \Bar{z})\right) \dd \Bar{z} .
\eestar
Similarly, we have the element $F(z,\zbar) \partial_z \in
\Omega^{0,*}(A,TA)$ and the formula above implies
\ben
\Bar{\partial}(F (z,\zbar)\partial_z) = L_m(A_1) - L_m(A_3) .
\een
Let $\dd$ denote the differential in $C_*(\scr{L}_\bfC
(A))$. The above implies
\ben
\dd(F(z,\Bar{z}) \partial_z \cdot L_n(A_2)) \; = \; (L_m(A_1) - L_m(A_3))
L_n(A_2) + \left\lfloor F(z,\Bar{z}) \partial_z, L_n(A_2)\right\rfloor_{\sVir_0(A_2)} .
\een
We compute
\bestar
\left\lfloor F(z,\Bar{z}) \partial_z, L_n(A_2)\right\rfloor_{\sVir_0(A_2)} & = &
f_2(r^2) \dd \Bar{z}  \left\lfloor z^{m+1}\partial_z,
z^{n+2}\partial_z\right\rfloor +
z^{m+n+3} \frac{\partial f_2(r^2)}{\partial z} \dd \zbar \partial_z \\ & = &
(m-n-1) L_{m+n}(A_2) + z^{m+n+3} \frac{\partial f_2(r^2)}{\partial z}
\dd \zbar \partial_z .
\eestar
Combining, obtain
\be\label{coh1}
L_m \bullet L_n - L_n \bullet L_m - [L_m,L_n] - L_{m+n}+
\left\lfloor z^{m+n+3}\frac{\partial f_2(r^2)}{\partial z} \dd
  \zbar \partial_z \right\rfloor = 0
\ee
where the bracket denotes the cohomology class. We consider the last
term. Introduce the element $z^{n+m+2} \Bar{z}
f_2(r^2) \partial_z$. Applying the $\dbar$-operator we get
\bestar
\Bar{\partial} (z^{n+m+2}\Bar{z} f_2(r^2) \partial_z) & = & z^{n+m+2} f_2(r^2) \dd
\Bar{z} \partial_z + z^{n+m+2} \Bar{z} \left(\frac{ \partial f_2(r^2)}{\partial\Bar{z}}\right) \dd
\Bar{z} \partial_z  \\ &
= & L_{n+m}(A_2) + z^{n+m+3} \frac{\partial f(r^2)}{\partial z} \dd \Bar{z} \partial_z
\eestar
where in the last line we use the fact that $\frac{\partial}{\partial
  \zbar} f_2(r^2) = z f_2'(r^2)$ and $\frac{\partial}{\partial
  zr} f_2(r^2) = \zbar f_2'(r^2)$. Thus, in cohomology we have
\ben
\left\lfloor z^{n+m+3} \frac{\partial f_2}{\partial z} \dd \zbar \partial_z
\right\rfloor = L_{n+m}
\een
so that (\ref{coh1}) simplifies to 
\ben
L_m \bullet L_n - L_n \bullet L_m - [L_m,L_n] = 0 .
\een

\subsection{The case of nonzero central charge}
We now describe the twisted case. As a vector space we have
\ben
{\rm Vir} = \bfC((z)) \partial_z \oplus \bfC \cdot C
\een
where $c$ is the central parameter. We recall that the Lie bracket is
\ben
[L_n,L_m] = (m-n) L_{n+m} + \frac{m^3-m}{12} \delta_{n,-m} c .
\een

Again, let $I \subset \bfR_{>0}$ and write $A = \rho^{-1}(I)$. The map
$\Phi$ is defined by $\Phi(I)|_{\bfC((z)) \partial_z} = \Phi_{\rm un}$
and it sends the central parameter of ${\rm Vir}$ to the central
parameter of $\scr{L}_c^\bfC(A) \oplus \bfC \cdot C[-1]$. 

The factorization algebra $\sVir$ assigns to the annulus $A$ the
dg vector space:
\ben
\sVir(A) = \left({\rm Sym}(\Omega^{0,*}(A,TA)[1] \oplus
  \bfC \cdot C), \dbar + \dd_{\rm CE}\right) .
\een
where $\omega \in C^1(\scr{L}^\bfC)$ is the central extension as
above. We need to show
\ben
L_m \bullet L_n- L_n \bullet L_m = (m-n) L_{n+m} + \frac{m^3 - m}{12}
\cdot c .
\een

Let the notation be as above. We have
\bestar
\dd (F(z,\Bar{z}) \partial_z \cdot L_n(A_2)) & = & (L_m(A_1) - L_m(A_3))
L_n(A_2) + \left[ F(z,\Bar{z}) \partial_z, L_n(A_2)\right]_{\Hat{\sL}(A_2)} \\ & = & (L_m(A_1) - L_m(A_3))
L_n(A_2) - (m-n-1) L_{m+n}(A_2) +
z^{m+n+3} \frac{\partial f_2}{\partial z} \dd \zbar \partial_z  \\ & -
&
\frac{1}{2 \pi} \frac{c}{12}
\int_{A} F(z,\Bar{z})\frac{\partial^3}{\partial z^3}\left(f_2(r^2) z^{n+2}\right) \dd z
\dd\Bar{z} .
\eestar
Everything is the same as the zero central charge calculation except
for the last line. Applying the same trick as in the previous section
to the second line, we see that $\dd
(F(z,\Bar{z}) \partial_z \cdot L_n(A_2))$ is cohomologous to 
\bestar
(L_m(A_1) - L_m(A_3))
L_n(A_2) & - & (m-n) L_{m+n}(A_2) \\ & - & \frac{1}{2 \pi}
\frac{c}{12} \int_A \frac{\partial^3}{\partial z^3} \left(F(z,\Bar{z})\right) f_2(r^2) z^{n+2} \d z
\dd\Bar{z} 
\eestar

We compute
\bestar
\int_A \frac{\partial^3}{\partial z^3} \left(F(z,\Bar{z})\right)
f_2(r^2) z^{n+2} \dd z
  \dd \Bar{z} & = &\int_A f_2(r^2) z^{n+2} \partial_z^3 \left(z^{m+1}
    \int_{0}^{z\Bar{z}} f_1(s) - f_3(s) \; \dd s \right) \dd z \dd
  \Bar{z} \\ & = &  (m^3-m) \int_A f_2(r^2)z^{n+ m}
\dd
  z \dd \Bar{z} \\ 
  & = & (m^3-m) \left(\int_{0}^{2\pi} e^{i(n+m) \theta} \dd \theta \right)
  \left(\int_{0}^r f_2(r^2) r^{n+m} r\dd r\right) \\ & = & 2 \pi  (m^3-m) \delta_{n,-m} .
\eestar
In the second line we used the fact that the function $z \mapsto
\int_0^{z \Bar{z}} f_1 -f_3$ is constant on $A_2$. Thus, $\dd(F(z,\zbar) \partial_z \cdot L_n(A_2))$ is cohomologous to 
\ben
(L_m(A_1) - L_m(A_3))
L_n(A_2) - (m-n) L_{m+n}(A_2) - \frac{m^3 - m}{12} \delta_{n,-m} \cdot
c .
\een
Wrapping everything up, in cohomology we have verified 
\ben
0 = L_m \bullet L_n - L_n \bullet L_m - [L_m,L_n] + \frac{m^3-m}{12}
\delta_{n,-m} \cdot c
\een
as desired.

This completes the proof of Proposition \ref{phi}. 

\section{The vertex algebra structure}\label{secvert}

We sketch the main points of Costello-Gwilliam's treatment of
extracting vertex algebras from structured factorization algebras on
$\bfC$. We then use their characterization to show that the
factorization algebra $\sVir$ determines a vertex algebra and
go further to identify it with the usual Virasoro vertex algebra using
the construction.

First, we need to review the definition of a vertex algebra. It consists of a vector
space $V$ over the field $\bfC$ along with the following data:
\begin{itemize}
\item A vacuum vector $\left|0\> \in V$.
\item A linear map $T : V \to V$ (the translation operator).
\item A linear map $Y(-,z) : V \to {\rm End}(V)\llbracket z^{\pm 1}
  \rrbracket$ (the vertex operator). We write $Y(v,z) = \sum_{n \in \bfZ} A_n^v z^{-n}$
  where $A_n^v \in {\rm End}(V)$. 
\end{itemize} 
satisfying the following axioms:
\begin{itemize}
\item For all $v,v' \in V$ there exists an $N \gg 0$ such that $A_n^v
  v' = 0$ for all $n > N$. (This says that $Y(v,z)$ is a {\it field}
  for all $v$). 
\item (vacuum axiom) $Y(\left| 0 \>, z) = {\rm id}_V$ and
    $Y(v,z)\left|0\> \in v + z V \llbracket z \rrbracket$ for all
    $v \in V$. 
\item (translation) $[T,Y(v,z)] = \partial_z Y(v,z)$ for all $v \in
  V$. Moreover $T$ kills the vacuum. 
\item (locality) For all $v,v' \in V$, there exists $N \gg 0$ such
  that 
\ben
(z-w)^N[Y(v,z),Y(v',w)] = 0
\een
in ${\rm End}(V) \llbracket z^{\pm 1},w^{\pm 1}\rrbracket$. 
\end{itemize}

We will utilize a reconstruction theorem for vertex
algebras. It says that a vertex algebra is completely and uniquely determined by a
countable set of vectors, together with a set of fields of the same
cardinality and a translation operator subject to a list
of axioms. 

\begin{theorem}[Theorem 2.3.11 of \cite{FBZ}] Let $V$ be a complex vector space equipped with: an
  element $\left|0 \> \in V$, a linear map $T : V \to V$, a countable
    set of vectors $\{a^s\}_{s \in S} \subset V$, and fields $A^s(z) =
    \sum_{n \in \bfZ} A_n^sz^{-n-1}$ for each $s\in S$ such that:
\begin{itemize}
\item For all $s \in S$, $A^s(z) \left|0\> \in a^s + z V\llbracket
    z\rrbracket$;
\item $T \left|0\> = 0$ and $[T,A^s(z)] = \partial_z A^s(z)$;
\item $A^s(z)$ are mutually local;
\item and $V$ is spanned by $\{A_{j_1}^{s_1} \cdots A_{j_m}^{s_m}
  \left|0\>\}$ as the $j_i's$ range over negative integers. 
\end{itemize}
Then, the data $(V,\left|0\>, T,Y)$ defines a unique vertex algebra satisfying 
\ben
Y(a^s,z) = A^s(z) .
\een
\end{theorem}

The main result of this section identifies two vertex algebras: the
first comes from the factorization algebra, the other one is the
Virasoro vacuum vertex algebra defined in the next section. We prove
these are the same using the above reconstruction theorem.

\subsection{The Virasoro vertex algebra} 
We recall the definition of the Virasoro vertex algebra. For us, it
will be a vertex algebra over the polynomial ring $\bfC[c]$. For an
arbitrary value of $c$ this will specialize to the usual Virasoro
vertex algebra associated to that central charge. First, consider, as
we did above, the associative algebra given by the universal envelope
of the Virasoro Lie algebra $U = U({\rm Vir})$. There is a subalgebra $U_+
\subset U({\rm Vir})$ generated by elements of the form $z^{n+1} \partial_z$ with
$n \geq -1$. Next, define
\ben 
{\bf Vir} = {\rm Ind}_{U_+}^U \; \bfC_c = U \tensor_{U_+} \bfC_c
\een
where the $L_n$'s act trivially on $\bfC_c$ and the central parameter $C$
acts by multiplication by $c$. The vacuum vector is the natural image
of the element $1 \tensor 1 \in U \tensor \bfC$ in ${\bf Vir}_c$. The fields are
\ben 
L(z) := \sum_{n \in \bfZ} L_n z^{-n-2} 
\een
and translation operator is $T = L_{-1} = \partial_z$. These satisfy the axioms
in the reconstruction theorem, and so define a vertex algebra, simply
denoted ${\bf Vir}$. We will call this $\bfC[c]$-linear vertex algebra
the {\em Virasoro vertex algebra}. Note that when we specialize to a particular complex
number we obtain the $\bfC$-linear vertex algebra ${\bf Vir}|_{c = c_0} = {\bf
  Vir}_{c_0}$ called the {\em Virasoro vertex algebra of central
  charge $c$}. 

\subsection{From factorization to vertex} 
In the first part of this note we studied a particular two-dimensional
factorization algebra and did not mention a vertex algebra. This
section is a bit of an aside and sketches the relationship between certain structured factorization algebras on $\bfC$ and vertex
algebras. This relationship is made more precise in \cite{CG1}, but we
try to sketch the main points. The main result is essentially a functor
from a  subcategory of factorization algebras on $\bfC$ to vertex algebras, and
we will use this result to read off the vertex
algebra structure from the factorization algebra $\sVir$
above. 

The maps $Y(-,z)$ encode the ``multiplication'' of the
vertex algebra. We can view it has a multiplication parametrized by a
complex coordinate $z \in \bfC$. Consider the two points $0,z \in \bfC$ with $z \ne
0$. This multiplication has the form
\ben
Y_z : V_0 \tensor V_z \to V((z)) 
\een
Critical to the structure of
a vertex algebra is holomorphicity. Indeed, the axioms imply that the
$Y_z$'s vary holomorphically. Thus, the factorization algebra we start
with must be translation invariant (so the vector space assigned
does not depend on the open set up to translations) together with a
holomorphicity condition. 

For the remainder of this section, let $\sF$ be a
prefactorization algebra on $\bfC$ in the appropriate category of
differentiable vector spaces. \footnote{Some care is needed to
  define this category correctly. We refer the interested reader to
  \cite{CG1}} 

We say that $\sF$ is {\it holomorphically translation invariant}
if
\begin{itemize}
\item $\sF$ is translation invariant. 
\item There exists a degree $-1$ derivation $\eta: \sF \to
  \sF$ such that $\dd \eta = \partial_{\Bar{z}}$ as derivations of
  $\sF$. 
\end{itemize}

Also important will be the notion of a {\it smooth $S^1$-equivariant}
structure on $\sF$. We will mention this shortly. For now, we
discuss how to read off the structure of a vertex algebra from a
holomorphic translation invariant factorization algebra. The key is
that such factorization algebra defines a coalgebra structure over a
certain (colored) cooperad.

Define the complex manifold
\ben
{\rm Discs}(r_1,\ldots,r_k) := \left\{z_1,\ldots,z_k \in \bfC \; | \;
D(z_1,r_1) \sqcup \cdots \sqcup D(z_k,r_k) \; {\rm disjoint}\right\}
\subset \bfC^k .
\een
The collection of these spaces form a $\bfR_{> 0}$-colored operad in the category of complex
manifolds, which we denote ${\rm Discs}$.  Applying the functor
$\Omega^{0,*}$ we get a $\bfR_{>0}$-colored cooperad $\Omega^{0,*}({\rm
  Discs})$ in the category of differentiable vector spaces. The main
technical fact that we use to read off the structure of a vertex
algebra is

\begin{prop}[\cite{CG1}] Let $\sF$ be a holomorphically
  translation-invariant factorization algebra on $\bfC$. Then,
  $\sF$ defines an algebra over the $\bfR_{>0}$-colored cooperad
  $\Omega^{0,*}({\rm Discs})$. 
\end{prop}

This means that at the level of cohomology as we let $p \in {\rm Discs}(r_1,\ldots,r_k)$
vary the factorization maps
\ben
m[p] : {\rm H}^*\sF(D(0,r_1)) \times \cdots \times {\rm H}^*\sF(D(0,r_k)) \to
{\rm H}^*\sF(\bfC)
\een
lift to a map
\ben
\mu_{z_1,\ldots,z_k}^{r_1,\ldots,r_k} : {\rm H}^*\sF(D(z_1,r_1)) \times \cdots \times {\rm
  H}^*\sF(D(z_k,r_k)) \to {\rm Hol}\left({\rm Discs}(r_1,\ldots,r_k),
{\rm H}^*\sF(\bfC)\right) .
\een
Translation invariance allows us to replace $\sF(D(z_i,r_i))
\simeq \sF(D(0,r_i))$ which we denote by $\sF(r_i)$, so we can
write this map as
\ben
\mu_{z_1,\ldots,z_k}^{r_1,\ldots,r_k} : \sF(r_1) \times \cdots \times \sF(r_k) \to {\rm Hol}\left({\rm Discs}(r_1,\ldots,r_k),
{\rm H}^*\sF(\bfC)\right) .
\een
Note that although the source space of this map does not depend on the
centers of the discs, the map itself does, hence the messy notation. 

For $r' < r$ the maps $\mu_{z_1,\ldots,z_k}^{r_1,\ldots,r_k}$ respect the natural inclusions
\ben
{\rm Discs}(r_1,\ldots,r_k) \hookrightarrow {\rm
  Discs}(r'_1,\ldots,r'_k)
\een
and so the limit of the multiplication map as $(r_1,\ldots,r_k) \to (0,\ldots,0)$ makes
sense and has the form
\ben
\mu_{z_1,\ldots,z_k} : \left(\lim_{r \to 0} {\rm H}^*(\sF(r))
\right)^{\tensor k} \to \lim_{r \to 0} {\rm Hol}\left({\rm
  Discs}_k(r), {\rm H}^*(\sF(r)) \right) \cong {\rm Hol}\left({\rm
Conf}_k(\bfC), {\rm H}^*\sF(\bfC)\right)
\een
where ${\rm Conf}_k(\bfC)$ is the ordered configuration space of
$k$-distinct points in $\bfC$. 

The last piece of data we need corresponds to the ``conformal
decomposition'' of a vetex algebra. For us, this will come from an
$S^1$-action on $\sF$. The reader is encouraged to look at
\cite{CG1} for a precise definition, but we assume that we have a {\it
  nice} action of $S^1$ on $\sF$ and it is compatible with the
translation invariance discussed above. 

We can now read off the data of the
vertex algebra from $\sF$:

\begin{itemize}
\item Let $\sF^{(l)}(r) \subset \sF(r)$ be the $l$th eigenspace for the $S^1$-action. The underlying vector
space for the vertex algebra is
\ben
V := \bigoplus_l {\rm H}^*(\sF^{(l)}(r)) .
\een 
\item The translation operator. The action of $\partial_z$ on $\sF^{(l)}(r)$ has the form
\ben
\partial_z : \sF^{(l)}(r) \to 
\sF^{(l-1)}(r) .
\een
We let $T : V \to V$ be the operator which is $\partial_z$ restricted
to the $l$-th eigenspace. 
\item The fields. Consider the map
\ben
\mu_{z,0} : \left(\lim_{r \to 0} {\rm
    H}^*(\sF(r))\right)^{\tensor 2} \to {\rm Hol}\left({\rm
  Conf}_2(\bfC), {\rm H}^*(\sF(\bfC))\right)
\een
defined above. Certainly, we have a map $V
  \to \lim_{r \to 0} {\rm
    H}^*(\sF(r))$, so it makes sense to resctict $\mu_{z,0}$ to a
  map 
\ben
V \tensor V \to {\rm Hol}\left({\rm Conf}_2(\bfC), {\rm
    H}^*(\sF(\bfC))\right) \simeq {\rm Hol}\left(\bfC^\times, {\rm
    H}^*(\sF(\bfC))\right) .
\een
Post composing this with the projection maps $H^*(\sF(\infty)) \to
V_l$ combine to define the map
\ben
\Bar{\mu}_{z,0} : V \tensor V \to \prod_l {\rm Hol}(\bfC^\times, V_l)
\een
We can perform Laurent expansions to view this as
\ben
\Bar{\mu}_{z,0} : V \tensor V \to \Bar{V} \llbracket z^{\pm 1}
\rrbracket .
\een
We define $Y(-,z) : V \to {\rm End}(V)\llbracket z^{\pm 1}\rrbracket$ by
\ben
Y(v,z) v' \; := \; \Bar{\mu}_{z,0}(v,v') .
\een
One can show that this actually lies in $V((z))$ for all $v,v'$. 
\end{itemize}

The above can be made much more precise and made into the following
theorem. 

\begin{theorem}[Theorem 5.2.2.1 \cite{CG1}] \label{fv} Let $\sF$ be a $S^1$-equivariant holomorphically translation invariant factorization algebra on $\bfC$. Suppose
\begin{itemize}
\item The action of $S^1$ on $\sF(r)$ extends smoothly to an action of the algebra of distributions on $S^1$. 
\item For $r < r'$ the map 
\ben
\sF^{(l)}(r) \to \sF^{(l)}(r')
\een
is a quasi-isomorphism.
\item The cohomology ${\rm H}^*(\sF^{(l)}(r))$ vanishes for $l \gg 0$.
\item For each $l$ and $r > 0$ we require that ${\rm H}^*(\sF^{(l)}(r))$ is isomorphic to a countable sequential colimit of finite dimensional vector spaces. 
\end{itemize}
Then $\mathbb{V}{\rm ert} (\sF) := \oplus_l {\rm H}^*(\sF^{(l)}(r))$ (which is independent of $r$ by assumption) has the structure of a vertex algebra.
\end{theorem}

Let ${\rm PreFact}_\bfC$ denote the category of prefactorization
algebras on $\bfC$. Let ${\rm PreFact}^{\rm hol}_\bfC \subset {\rm
  PreFact}_\bfC$ be the full subcategory spanned by prefactorization
algebras satisfying the conditions of the above theorem. This result
can be upgraded to provide a functor
\ben
\mathbb{V}{\rm ert} : {\rm PreFact}^{\rm hol}_\bfC \to {\rm Vert}
\een
where ${\rm Vert}$ is the category of vertex algebras. 

\subsection{Verifying the axioms}
In this section we verify the Virasoro factorization algebra $\sVir$
indeed satisfies the conditions of Theorem \ref{fv} necessary to determine a vertex
algebra stated in the last section. 


Explicitly, we show the following:
\begin{itemize}
\item[(1)] There is a $S^1$-action on $\sVir$ covering the
  action of $\bfC^\times$ by rotations. Moreover, for all $r > 0$ (including $r = \infty$) the $S^1$-action
  on $\sVir(r)$ extends to an action of $\mathcal{D}(S^1)$
  the space of smooth distributions on the circle. 
\item[(2)] Then for all $l$ and all $r < r'$ the
  natural map
\ben
\sVir^{(l)}(r) \to \sVir^{(l)}(r')
\een 
is an equivalence.
\item[(3)] ${\rm H}^*(\sVir^{(l)}(r)) = 0$ for $l \gg 0$.   
\item[(4)] The space ${\rm H}^*(\sVir^{(l)}(r))$ is a colimit of
  finite dimensional vector spaces for all $l,r$. 
\end{itemize}
The first condition is clear: the $S^1$-action comes from its
natural action on $\Omega^{0,*}_c(\bfC)$. We extend this to
distriutions $\varphi \in \mathcal{D}(S^1)$ by the rule
\ben
(\varphi \cdot \alpha)(z) = \int_{t \in S^1} \varphi(t) \alpha(t z)
\een
where $\alpha \in \Omega^{0,*}_c(\bfC)$. This extends naturally to
vector fields. 

Let's consider (2). For simplicity we work with the
(untwisted) factorization algebra $\sVir_0 = C_*(\scr{L}_\bfC)$, the
twisted case is similar. Consider the filtration of
$\sVir_0$ by symmetric tensor degree. Namely
\ben
F^m \sVir_0(r) = {\rm Sym}^{\leq m} (\scr{L}(D(0,r)[1]) = \bigoplus_{j \leq
  m} \left(\scr{L}(D(0,r))[1]^{\tensor j}\right)_{\Sigma_j} .
\een
The associated graded of this filtration is
\ben
{\rm Gr}^m \sVir_0(r) = \left(\scr{L}(D(0,r))[1]^{\tensor
    m}\right)_{\Sigma_m} 
\een
and there is a spectral sequence
\ben
{\rm H}^*({\rm Gr}^*\sVir_0(r)) \Rightarrow {\rm H}^*(\sVir_0(r)) .
\een
The filtration respects the $S^1$-action, so for each $l$ we get a
spectral sequence for the eigenspaces
\ben
{\rm H}^*({\rm Gr}^*\sVir_0^{(l)}(r)) \Rightarrow {\rm H}^*(\sVir_0^{(l)}(r)) .
\een
Thus, to verify that $\sVir_0^{(l)}(r) \to \sVir_0^{(l)}(s)$ is an equivalence
for $r < s$ it it enough to show that it is at the level of associated
gradeds. That is, we need to show that the restriction of the map
\ben
\Omega_c^{0,*}\left(D(0,r)^{m}, TD(0,r)^{\boxtimes m}\right) \to
\Omega_c^{0,*}\left(D(0,s)^{m}, TD(0,s)^{\boxtimes m}\right)
\een
to the $l$-eigenspaces is an equivalence. Again, we recall Serre's
result that for any open $U\subset \bfC$
\ben
{\rm H}_{\dbar}^*\left(\Omega^{0,*}_c(U,TU) \right) \; \cong \;
\left(\Omega_{\rm hol}^1(U,TU) \right)^\vee
\een
concentrated in degree $0$. When $U = D(0,r)$ we have a coordinization
\ben
\Omega^1_{\rm hol}(D(0,r)) = \bfC[z] \dd z .
\een
Now, $z^k$ has $S^1$-weight $k$. Thus $(z^k)^\vee$ has weight
$-k$. The weight of $(\dd z)^\vee$ is $-1$ and the weight of
$\partial_z^\vee$ is $+1$. This shows that the weight spaces are
independent of the radius chosen, so we have verified (2). Moreover, the weight spaces are
clearly finite dimensional and vanish for $m \geq 0$, so we also get
(3) and (4). 

Finally, Theorem \ref{fv} implies the following. 

\begin{prop} The $\bfC[c]$-module $V = \bigoplus_l {\rm
    H}^*(\sVir^{(l)}(r))$ has the
  structure of the vertex algebra (in $\bfC[c]$-modules) induced from the factorization
  structure on $\sVir$. In particular, for each $c \in \bfC$ the
  vector space $V_c = \bigoplus_l {\rm H}^*(\sVir_c^{(l)}(r))$ has the
  structure of a vertex algebra. 
\end{prop} 

\subsection{An isomorphism of vertex algebras}
The map $\Phi : U({\rm Vir}) \to {\rm
  H}^*(\sVir(A(r,r')))$ from Proposition \ref{phi} applied to the
interval $I = (r,r')$ gives $V$ the structure of a $U({\rm
  Vir})$-module. More precisely, let $\epsilon < r < R$ then we have a
factorization map
\ben
\sVir(D(0,\epsilon)) \tensor \sVir(A(r,R)) \to
\sVir(D(0,R)) .
\een
We have the following diagram
\ben
\xym{
{\rm H}^*\sVir(D(0,\epsilon)) \tensor {\rm H}^* \sVir(A(r,R))
\ar[r] & {\rm H}^*\sVir(D(0,R)) \\
V \tensor {\rm H}^*\sVir(A(r,R)) \ar[u] & V \ar[u] \\
V \tensor U({\rm Vir}) \ar[u]^-{1 \tensor \Phi_c} \ar@{.>}[ur]  .
}
\een
The top left arrow comes from the inclusion $V \hookrightarrow {\rm
  H}^*\left(\sVir(D(0,\epsilon) )\right)$. The dotted map exists since the
image of the factorization product on $V$, where we only see
finite sums of $S^1$-eigenvectors, still only contains finite sums of
$S^1$-eigenvectors. 

Our main result is:

\begin{theorem} There is a $\bfC[c]$-linear isomorphism of $U({\rm
    Vir})$-modules $\Psi :
  {\bf Vir} \to V$ which sends $\left| 0 \> \in {\bf Vir}_c$ to $1
    \in V$. It extends to an isomorphism of vertex algebras
\ben
\xymatrix{
\Psi : {\bf Vir} \ar[r]^-{\cong} & \VVert(\sVir)
}
\een
over the ring $\bfC[c]$. In particular, when we specialize to a $c \in \bfC$ we
    obtain an isomorphism of vertex algebras
\ben \xymatrix{
\Psi_c : {\bf Vir}_c \ar[r]^-{\cong} & \VVert(\sVir_c) .}
\een
\end{theorem}

\begin{proof} Recall that the vacuum vector is the image of $1$ under the map 
\ben
\xym{
U ({\rm Vir}) \ar[r]^{\id \tensor 1} & U({\rm Vir}) \tensor \bfC
\ar[r] & U({\rm Vir})
\tensor_{U({\rm Vir})_+} \bfC_c = {\bf Vir}_c .}
\een
We define the map of $U({\rm Vir})$-modules
\ben
U({\rm Vir}) \tensor \bfC \to V
\een
by sending $1 \tensor 1$ to $1$ and extending by $U({\rm
  Vir})$-linearity. We need to check that this descends to ${\bf
  Vir}_c$. That is, we verify that
  $1 \in V$ is killed by $L_n$ for $n \geq -1$. Recall $L_n(A) = f(z\Bar{z}) z^{n+2}
  \dd \Bar{z} \partial$ is a representative for $L_n$ on
  $\scr{L}^\bfC(A(r,R))$ where $f(z \Bar{z})$ is a bump function as
  above. It suffices to show that $L_n(A)$ is exact when viewed as an
  element in $\scr{L}^\bfC(D(0,R))$. Define $h(z,\Bar{z}) := \int_{z
    \Bar{z}}^\infty f(s) \; \dd s$ and note that the chain rule implies
\ben
 \Bar{\partial}(h(z,\Bar{z}) \; z^{n+1}) \; = \; f(z \Bar{z}) z^{n+2} \dd \Bar{z} 
\een 
Thus, $L_n(A)$ is exact via the element $h(z,\Bar{z})
z^{n+1} \partial$. This shows that we get a well-defined map ${\bf
  Vir}_c \to V$ that sends $\left|0 \right> \mapsto 1$. 

We need to see that this map is an isomorphism of $U({\rm
  Vir})$-modules. We take advantage of some filtrations. Consider the natural filtration of the tensor
algebra of ${\rm Vir}$, namely
\ben
F^i({\rm Tens}({\rm Vir})) := \tensor_{j \geq
  i} {\rm Vir} .
\een
This descends to a filtration on $U({\rm Vir})$. Similarly, define the filtration of
$\sVir$ by
\ben
F^i \sVir(U) = {\rm Sym}^{\leq i}
(\Omega^{0,*}_c(U,TU)[1] \oplus \bfC \cdot C)
\een
with induced differential. It is clear that this is a subcomplex and
hence descends to cohomology. Moreover the map of modules defined
above respects both of these filtrations. 

With respect to the above filtration we have the identification
\ben
{\rm Gr} \; {\bf Vir} \cong {\rm Sym}^*(z^{-1}\bfC[z^{-1}]\partial_z \oplus \bfC \cdot C) =
{\rm Sym}^*(z^{-1}\bfC[z^{-1}]\partial_z)[C] .
\een

Moreover, we have the identifications of associated gradeds
\ben
{\rm Gr} \; U({\rm Vir}) \; = \; {\rm Sym}({\rm Vir}) = {\rm
  Sym}(\bfC[z,z^{-1}]\partial_z)[C] \;\; {\rm and}
\;\; {\rm Gr} \; \sVir(U) = \Hat{\rm Sym}
(\Omega^{0,*}_c(U,TU) [1] \oplus \bfC \cdot C) .
\een
Consider the map $U({\rm Vir}) \to V$ induced by the action of $U({\rm
  Vir})$ on the unit $1 \in V$. We have the diagram of associated gradeds
\ben
\xym{
{\rm Gr} \; U({\rm Vir}) \ar[r]\ar@{=}[d] & {\rm Gr} \; V \ar@{^{(}->}[r] & {\rm
  H}^*\left({\rm Sym}^*(\Omega^{0,*}_c(D(0,r),TD(0,r))[1] \oplus \bfC
  \cdot C) \right) \\
 {\rm
  Sym}(\bfC[z,z^{-1}]\partial_z)[C] \ar[ur] & 
}
\een
The embedding on the right is the direct sum of $S^1$-eigenspaces and
is identified with $z^{-1}\bfC[z^{-1}][C]$. Thus, the map ${\rm Gr} \;
U({\rm Vir}) \to {\rm Gr} \; V$ is the map of commutative algebras
\ben
 {\rm
  Sym}(\bfC[z,z^{-1}]\partial_z)[C] \to  {\rm
  Sym}(z^{-1}\bfC[z^{-1}]\partial_z)[C]
\een
and is induced by natural map $\bfC[z,z^{-1}] \to z^{-1}
\bfC[z^{-1}]$. 

Concluding we see that the map ${\rm Gr} \; {\bf Vir} \to {\rm Gr}
\; V$ is an isomorphism, and since there are no extension problems
over $\bfC[c]$ we have the desired isomorphism of $U({\rm
  Vir})$-modules. 

Finally, we need to show that the OPE's agree so that the module
isomorphism extends to an isomorphism of vertex algebras. Namely, we will
show
\ben
m_{z,0}(L_{-2}\cdot 1, v) \; = \; \sum_ {n \in \bfZ} (L_n \cdot v) z^{-n-2} .
\een
Now, the residue pairing allows us to represent $L_n(A(r,R))$ by the
linear map
\ben
\Omega_{\rm hol}^1(A(r,R)) \to \bfC \;\; , \;\; h(z) \dd z \mapsto
\left(\oint_{S^1} z^{n+1} h(z) \dd z \right)  .
\een

Fix a point $z_0 \in A(r,R)$. By Cauchy's theorem we have for some $\epsilon$ such that $\epsilon <
|z_0| - r$ and $\epsilon < s - |z_0|$:  
\ben
2 \pi i h(z_0) = \oint_{|\zeta| = R-\epsilon} \frac{h(\zeta)}{\zeta-z_0} \dd \zeta -
\oint_{|\zeta| = r + \epsilon} \frac{h(\zeta)}{\zeta-z_0} \dd \zeta .
\een
For the first integral we have $|z_0| < |\zeta|$ and we can expand
\ben
\frac{1}{\zeta-z_0} = \frac{1}{\zeta} \cdot \frac{1}{1-\frac{z_0}{\zeta}} =
\frac{1}{\zeta} \sum_{j=0}^\infty \left(\frac{z_0}{\zeta}\right)^j =
\sum_{j=0}^\infty z_0^j \zeta^{-j-1} .
\een
Thus 
\ben
\oint_{|\zeta| = R-\epsilon} \frac{h(\zeta)}{\zeta-z_0} \dd \zeta = \sum_{j=0}^\infty
\left(\oint_{|\zeta| = R- \epsilon} h(\zeta) \zeta^{-j-1} \dd \zeta\right) z_0^j .
\een
Similarly the second integral can be written as
\ben
\oint_{|\zeta| = r+\epsilon} \frac{h(\zeta)}{\zeta-z_0} \dd \zeta = -\sum_{j=0}^\infty
\left(\oint_{|\zeta|=r+\epsilon} h(\zeta) \zeta^j\right)z_0^{-j-1} .
\een
Since $h$ is holomorphic on $A(r,R)$ we can combine these integrals
by choosing a common contour and reindexing to write
\ben
\sum_{j=0}^\infty
\left(\oint_{|\zeta| = R- \epsilon} h(\zeta) \zeta^{-j-1} \dd \zeta\right) z_0^j  + \sum_{j=0}^\infty
\left(\oint_{|\zeta|=r+\epsilon} h(\zeta) \zeta^j\right)z_0^{-j-1} = \sum_{n \in
  \bbZ} \left(\oint \zeta^{n+1}  h(\zeta) \dd \zeta\right) z_0^{-n-2}  .
\een
This completes the proof. 
\end{proof}

\section{Universal factorization algebras and the Virasoro} \label{global}
It is a natural to extend the Virasoro factorization algebra to
general one-dimensional complex manifolds. Moreover, from the point of
view of of conformal field theory, \cite{Segal} for instance, it is
essential to consider this a global version of the Virasoro algebra
defined on general Riemann surfaces. Vertex algebras are of course
local in nature, from above they correspond to factorization on
$\bfC$. In this section we transition to studying a version of the
Virasoro factorization algebra defined on a general one-dimensional
complex manifold. 

One approach would be to construct a factorization algebra on each
manifold independently. It is convenient for us, however, to consider
the {\it site} of complex manifolds. Define the category ${\rm
Hol}_1$ whose objects are one-dimensional complex manifolds and whose maps are holomorphic embeddings. This is a
symmetric monoidal category with respect to disjoint union
$\sqcup$. Just as in the case of a fixed manifold, Weiss covers define
a Grothendieck topology on ${\rm Hol}_1$. 

\begin{defn}\label{defn universal}
A {\em universal} holomorphic prefactorization algebra (valued in the
category $\dgNuc^{\tensor}$) is a symmetric monoidal functor
\ben
{\rm Hol}_1^{\sqcup} \to \dgNuc^{\tensor} .
\een 
A {\em universal} holomorphic factorization algebra is a universal
holomorphic prefactorization algebra satisfying descent for Weiss
covers.
\end{defn}

\begin{remark} The term {\em universal} has appeared in the literature of
  vertex algebras and their close relatives, chiral algebras and we'd
  like to point out how our terminology is different. In Section 3.4.14 of \cite{BD} the term {\em universal chiral
  algebra} is used to refer to chiral algebras that are valued in the category of
modules for the Harish-Chandra pair $({\rm Aut}(\Hat{D}), {\rm W}_1)$ of
formal automorphisms and formal derivations of the holomorphic
disk. In Section 6.3 of \cite{FBZ} such a structure in
the category of vertex algebras is referred to as a {\em quasi-conformal}
vertex algebra. We stress that this is different than the notion of
universal considered in Definition \ref{defn universal}.
One can also realize the analog of a quasi-conformal structure in the setting of holomorphic
factorization algebras by factorization algebras valued in the
category of $({\rm Aut}(\Hat{D}), {\rm W}_1)$ modules. 
\end{remark}

We can produce such universal holomorphic factorization algebras from sheaves of Lie
algebras on the site ${\rm Hol}_1$. Indeed, given a sheaf of Lie
algebras $\scr{G}$ we can apply the Chevalley-Eilenberg chains functor
applied to compactly supported sections
$C_*(\scr{G}_c)$ to get a universal factorization algebra. Moreover,
this functor satisfies descent so that it defines a universal
holomorphic factorization algebra. We will denote this universal
factorization algebra by $U^{fact} \scr{G}$. 

\begin{example} 
Let us consider a fundamental example of a universal holomorphic
factorization algebra. Fix an ordinary Lie algebra $\fg$ and define the sheaf of Lie algebras on
${\rm Hol}_1$ by sending the complex one-manifold $\Sigma$ to the dg
Lie algebra $\fg^\Sigma := \Omega^{0,*}(\Sigma ; \fg)$. The differential is given by $\dbar \tensor 1_{\fg}$ and the Lie
bracket extends that of $\fg$. In doing so, one obtains the universal
factorization algebra $U^{fact} \fg^{(-)}$ that sends $\Sigma \mapsto C_*(\fg^\Sigma)$. If $\fg$ has a invariant pairing
$\<-,-\>_\fg$ one can use the cocycle on $\fg^\Sigma$
defined by 
\ben
(\alpha,\beta) \mapsto \int_\Sigma \<\alpha \wedge \partial \beta\>_\fg
\een
to define a central extension $\Hat{\fg}^\Sigma$. One obtains
a universal factorization algebra via 
\ben
U^{fact} \Hat{\fg} : \Sigma \mapsto C_*(\Hat{\fg}^\Sigma) .
\een
This is the unverisal factorization algebra representing the Kac-Moody
vertex algebra, see Chapter 5 of \cite{CG1}.
\end{example}

We will produce the universal
Virasoro factorization algebra in the same manner. Indeed, for each
$\Sigma$ in ${\rm Hol}_1$ we
have the dg Lie algebra 
\ben
\scr{L}^\Sigma = \Omega^{0,*}(\Sigma,T_\Sigma)
\een 
with differential given by $\dbar$ and bracket given by extending the
usual Lie bracket of holomorphic vector fields. The assignment
$\scr{L} : \Sigma \mapsto \scr{L}^\Sigma$ defines a symmetric monoidal
functor from the category ${\rm Hol}_1$ to the category of dg Lie
algebras with symmetric monoidal structure given by direct sum (of
underlying graded vector spaces). As the functor of
Chevalley-Eilenberg chains $C_*(-)$ is symmetric monoidal we get a symmetric monoial functor given by the {\it universal} envelope of $\scr{L}$
\ben
U^{fact}\scr{L} : {\rm Hol}_1 \to \dgNuc^{\tensor} \;\; , \;\; \Sigma \mapsto C_*(\scr{L}^\Sigma_c) .
\een
Applied to $\Sigma = \bfC$, of course we are in the situation of the
previous portion of the paper. 

The interesting part from the point of view of conformal field theory
and representation theory is the envelope of a {\it central extension}
of the sheaf of dg Lie alegbras $\scr{L}$. There is a potential
problem defining this central extension based on our formula given in
Section 2. Indeed, the cocycle on $\scr{L}^\bfC$
\ben
\omega(\alpha \tensor \partial_z, \beta \tensor \partial_z) = \frac{1}{2\pi}\frac{1}{12}\int_U \left(
  \partial_z^3 \alpha_0 \beta_1 +
  \partial_z^3\alpha_1 \beta_0
\right) \d^2 z 
\een
clearly depends on the choice of a coordinate (its failure to be
coordinate independent is precisely measured by the Schwarzian). Thus, there is no obvious way of
constructing a universal twisted envelope on all holomorphic
one-manifolds simultaneously. 

\subsection{First fix: uniformization}

A Riemann surface is a complex manifold of dimension one. Therefore,
it is given by a covering $\{U_i\}$ such that all transition functions
are holomorphic diffeomorphisms. The cocycle $\omega$ is not invariant
under arbitrary diffeomorphisms: if $w = f(z)$ it is not necessarily true that $f^* (\omega_z) = \omega_w$. 

One way of formulating the uniformization theorem for Riemann surfaces is that one can always find a subordinate cover to $\{U_i\}$ such
that the transition functions have the form
\ben
w = f(z) = \frac{a z + b}{c z + d}
\een
with $ad - bc \ne 0$. I.e., we can reduce to the projective linear structure group. Let
${\rm Hol}_1^{\rm proj} \subset {\rm Hol}_1$ denote the full subcategory of covers
where the transition functions are projective. The above says that
there is a section ${\rm unif} : {\rm Hol}_1 \to {\rm Hol}_1^{\rm proj}$ of the inclusion
\ben
{\rm Hol}_1 \hookrightarrow {\rm Hol}_1^{\rm proj} .
\een 

\begin{lemma} The cocycle $\omega$ is invariant under projective changes of coordinate. That is, for $f$ a projective diffeomorphism one has $f^* \omega =
  \omega$.
\end{lemma}

Thus, we can form a factorization algebra $\mathcal{F}_{\omega}^\Vir$ on ${\rm
  Hol}_1^{\rm proj}$. Using the uniformization construction this pulls back to a
factorization algebra on Riemann surfaces via
\ben
\xymatrix{
{\rm Hol}_1 \ar@{.>}[r]^-{\rm unif} & {\rm Hol}^{\rm proj}_1 \ar[r]^-{\mathcal{F}^{\Vir}_\omega}
& \dgNuc .
}
\een 
The problem with this construction is that the induced extension cocycle is not so obvious to write down. There is a more explicit way of doing this. 

\subsection{Second fix: projective connections}

We recall Atiayh's \cite{At1} formulation of connections on holomorphic vector bundles. Let $E$ be a holomorphic vector bundle on a complex manifold $X$. Denote by ${\rm Diff}^{\leq 1}(E) \subset {\rm Diff}(E)$ the subspace of order one differential operators on $E$. There is a short exact
sequence of vector bundles
\ben
0 \to {\rm End}(E) \to {\rm Diff}^{\leq 1} (E) \to T^{1,0}_X \tensor
{\rm End}(E) \to 0 
\een 
where the last map is the symbol map of an order one differential operator. Form
the pull-back along the inclusion of $T_X^{1,0} \hookrightarrow
T_X^{1,0} \tensor {\rm End}(E)$ via $x \mapsto x \tensor {\rm
  id}$. The resulting bundle is the {\it Atiyah}-bundle
\ben
0 \to {\rm End}(E) \to {\rm At}(E) \to T_X^{1,0} \to 0 .
\een 
Atiyah showed that splittings of this sequence are precisely
holomorphic connections. 

Consider the inclusion $\mathcal{O}_X \hookrightarrow {\rm End}(E)$ by viewing
$s \mapsto f \cdot s$ for $f \in \mathcal{O}_X$. One gets the induces sequence
\ben
0 \to {\rm End}(E) / \mathcal{O}_X \to {\rm At}(E) / \mathcal{O}_X \to T_X^{1,0} \to 0 .
\een 
By definition, {\it projective connections} are splittings of the
above sequence. 

\begin{itemize}
\item Non-trivial holomorphic connections on $T_\Sigma$ exist only in genus $1$,
  this is a consequence of Riemann-Roch.
\item Projective connections on $T_\Sigma$ exist for all Riemann surfaces and form a
  torsor over quadratic holomorphic differentials $\Omega^1_{\rm
    hol}(\Sigma)^{\tensor 2}$. 
\end{itemize}

Let ${\rm Hol}_1^{\nabla}$ denote the category of pairs
$(\Sigma,\nabla)$ where $\nabla$ is a projective connection for the
holomorphic tangent bundle $T_\Sigma^{1,0}$. There is
a forgetful functor 
\ben
\pi : {\rm Hol}_1^\nabla \to {\rm Hol}_1 
\een 
that we should think of as a $(\Omega^1_{hol})^{\tensor 2}$-torsor.

Fix a projective connection $\nabla$ on $\Sigma$. Locally, on $\Sigma$
consider the bilinear on $\scr{L}_c(U_z)$
\ben
\omega_{\nabla,z} (X,Y) = \omega_z(X,Y) + \nabla_z \cdot [X,Y] .
\een 

\begin{prop} 
\begin{itemize} 
\item $\omega_\nabla$ defines a cocycle on $\scr{L}_c(U_z)$ and is invariant under holomorphic changes of coordinate.
\item If $\nabla'$ is another projective connection we have
  $\omega_\nabla \sim \omega_{\nabla'}$. 
\end{itemize}
\end{prop}
\begin{proof}
{\bf Coordinate invariance} In writing down $\omega_{\nabla,z} = \omega_z$
  we have used a coordinate. We check coordinate invariance, so that
  it defines a section over $\Sigma$. It suffices to understand the case $U_z = \bfC_z$. Suppose
  $f : \bfC_w \to \bfC_z$ is a change of coordinates. Let $u :
  \scr{L}^\bfC \tensor \scr{L}^\bfC \to \Omega^{1,1}_{\bfC}$ denote
  the bilinear map
\ben
(\alpha \partial_z, \beta \partial_z) \mapsto \left(
  \partial_z^3 \alpha_0 \beta_1 - \alpha_0 \partial_z^3 \beta_1) +
  (\partial_z^3\alpha_1 \beta_0 - \alpha_1 \partial_z^3 \beta_0) 
\right) \dd \Bar{z} \dd z .
\een
We compute the difference
\ben
f^*u(\alpha \partial_z, \beta \partial_z) -
u\left(f^*(\alpha \partial_z), f^*(\beta \partial_z)\right) = 2
\left[(\alpha_0 \partial_w \beta_1 - \partial_w \alpha_0 \beta_1)
  + (\alpha_1 \partial_w \beta_0 - \partial_w\alpha_1 \beta_0)
\right] S(f) \Bar{\left(\frac{\partial f}{\partial w}\right)} \dd w
\dd \Bar{w} 
\een
where $S(f)$ is the holomorphic function called the
Schwarzian. Explicitly it is given in terms of first, second, and
third holomorphic derivatives of $f$:
\ben
S(f)(z) = \frac{\partial}{\partial z} \left( \frac{\partial^2 f
    / \partial z^2}{\partial f / \partial z} \right) - \frac{1}{2} \left( \frac{\partial^2 f
    / \partial z^2}{\partial f / \partial z} \right)^2 .
\een 
So the failure of the cocycle $u$ to be independent of a choice of
coordinate is measured by the Schwarzian. 

Let $P : \scr{L}^\bfC \tensor \scr{L}^\bfC \to \Omega^{1,1}(\bfC)$ be
the bilinear 
\ben
(\alpha \partial_z, \beta \partial_z) \mapsto \rho_z \cdot \left( (\alpha_0 \partial_z
  \beta_1 - \partial_z \alpha_0 \beta_1) + (\alpha_1 \partial_z
  \beta_0 - \partial_z \alpha_1 \beta_0) \right) \dd z \dd \Bar{z} .
\een
We compute the difference
\ben
f^*P(\alpha \partial_z, \beta \partial_z) - P\left(f^*(\alpha \partial_z),
f^*(\beta \partial_z) \right) =  \left( (\alpha_0 \partial_z
  \beta_1 - \partial_z \alpha_0 \beta_1) + (\alpha_1 \partial_z
  \beta_0 - \partial_z \alpha_1 \beta_0) \right) S(f)
\Bar{\left(\frac{\partial f}{\partial w}\right)} \dd w \dd \Bar{w} .
\een
This shows that the bilinear $u + 2 P$ is independent of
choice of coordinates. Finally, note that 
\ben
\omega_\bfC = \int_\bfC \; \circ \;  (u + 2 P)
\een
is the desired cocycle defining the extension of $\scr{L}^\bfC$, so we are done. 

{\bf Cocycle condition} We need to show that $\omega_U$ is a cocycle
for the Lie algebra $\scr{L}^\Sigma(U)$ for all $U$. We suppose $U \simeq \bfC$ and we check $\omega_\bfC$
is a cocycle. For simplicity write elements $\alpha \partial \in
\scr{L}^\bfC(\bfC)$ as $\alpha$. In terms of the bilinears
$u,P$ above we have
\bestar
\omega_\bfC ([\alpha,\beta],\gamma) +
\omega_\bfC([\beta,\gamma],\alpha) +
\omega_\bfC([\gamma,\alpha],\beta) & = & \int_\bfC\left[ \left(u([\alpha,\beta],
  \gamma) + u([\beta,\gamma],\alpha) + u([\gamma,\alpha], \beta)
\right) \right .\\ & + & 2\left. \left(P([\alpha, \beta],\gamma) + P([\beta,\gamma],\alpha) +
P([\gamma,\alpha],\beta) \right) \right]
\eestar
It follows from Jacobi that $P$-terms vanish. So, it
suffices to show that 
\ben
\int_\bfC \left(u([\alpha,\beta],
  \gamma) + u([\beta,\gamma],\alpha) + u([\gamma,\alpha], \beta)
\right) = 0 .
\een
This is a straightforward calculation.

Now, we show independence of $\omega_{\nabla,z}$ on a projective connection. Again, this is a local calculation. Suppose $\nabla,\nabla'$ are
  two projective connections, and let $\omega, \omega'$ and $P,P'$
  denote induced the bilinears as above, respectfully. We
  need to show that $\omega - \omega'$ is a coboundary
  when viewed as a cocycle in $C^*_{\rm red}(\scr{L}^\bfC)$. As mentioned above, the difference of
  two ordinary projective connections is simply a quadratic differential. It
  follows that we may view the difference $\nabla - \nabla'$ as an element
  in $\Omega^{0,*}(\Sigma, K_\Sigma^{\tensor 2})$. Then, we see that
  for $X \in \scr{L}^\Sigma$ 
\ben
(\nabla - \nabla') \cdot X = \<\nabla-\nabla',X\>
\een
where $\<-,-\>$ denotes the natural pairing
\ben
\Omega^{0,*}(\Sigma,K_\Sigma^{\tensor 2}) \tensor
\Omega^{0,*}(\Sigma,T_\Sigma) \to \Omega^{0,*}(\Sigma,K_\Sigma) \simeq
\Omega^{1,*}(\Sigma) .
\een
Denote by $\Phi = \int \circ \<\nabla-\nabla
',-\>_{(1,1)} : \scr{L}^\bfC \to
\bfC$. Note that $\Phi$ is linear of degree $-1$, so that it is a
$0$-cocycle for $\scr{L}^\bfC$. We have
\ben
(\omega - \omega')(\alpha,\beta) = \Phi([\alpha,\beta]) .
\een
This is what we wanted to show. 
\end{proof}

We take away two main observations: (1) there is a local cocycle $\omega \in {\rm
  H}_{\rm loc}^1(\sL^\Sigma)$ for each $\Sigma$ and hence an associated factorization algebra
$\Tilde{\sVir}$ on ${\rm Hol}_1^\nabla$ and (2) that we can descend along $\pi$ to get a factorization algebra
$\sVir : {\rm Hol}_1 \to \dgNuc$ as desired. When restricted to the over category of open sets $U \subset \bfC$ we
produce the factorization algebra from the first part of the paper,
hence the repetition of notation. 

The prefactorization algebra $\sVir$ can be
described explicitly as follows. To a pair $(\Sigma, \nabla)$ of a Riemann surface together with a projective connection we define
\ben
\Tilde{\sVir} (\Sigma,\nabla) = U^{fact}\Hat{\scr{L}}^{\Sigma} = U^{fact}_{\omega_\Delta} \scr{L}^\Sigma
\een
where $\Hat{\scr{L}}^\Sigma$ is the dg Lie algebra that is the extension of $\scr{L}^\Sigma$ determined by the cocycle $\omega_\nabla$. 

Thus, the factorization algebra $\sVir$ on ${\rm Hol}_1$ has
the following interpretation. Given a Riemann surface $\Sigma$ choose
{\it any} projective connection $\nabla$. Form the twisted envelope as
above. By the proposition this extension is independent of the
projective connection chosen. 

\subsection{Fixed Riemann surface} \label{sec fixed}
For each Riemann surface $\Sigma$ we can restrict our factorization algebra $\sVir$ to the overcategory ${\rm Hol}_{1 / \Sigma}$ to get a factorization algebra on $\Sigma$ which we denote $\sVir^\Sigma$. The construction depends on the 2-cocycle $\omega$, but on a fixed Riemann surface the choice is up to a scaling. 

\begin{prop}\label{prop1} Let $\Sigma$ be any Riemann surface. Then, we have
\ben
{\rm H}^1\left(C^*_{\rm loc}(\scr{L}^\Sigma)\right) \; \cong \; \bfC.
\een
\end{prop}
\begin{proof} 

Consider $\Sigma = \bfC$. Then
\ben
C^*_{\rm loc}(\scr{L}_\bfC) \; = \; \Omega^*_{\bfC}
\tensor_{\scr{D}_\bfC} C^*_{\rm red}\left({\rm
    Jet}_{\scr{L}_\bfC}\right) \; \cong \; \Omega^*\left(\bfC; C^*_{\rm red}({\rm
  Jet}_{\scr{L}_\bfC})\right)[2]
\een 
Now, ${\rm Jet}_{\scr{L}_\bfC}$ corresponds to the dg Lie algebra
$\bfC\llbracket z,\Bar{z}, \dd \Bar{z} \rrbracket \partial_z$ with
differential given by $\Bar{\partial}$. This is quasi-isomorphic to
the Lie algebra of formal holomorphic vector fields $W_1 := \bfC\llbracket z
\rrbracket \partial$ (with zero differential). So, we see $C^*_{\rm
  loc}(\scr{L}_\bfC) \simeq C^*(W_1)[2]$. A calculation of
Gelfand-Fuchs \cite{GF} implies that 
\ben
{\rm H}_{\rm red}^*(W_1) = \bfC[-3]
\een
concentrated in degree 3. The generator is of the form
$\partial_z^\vee \cdot (z \partial_z)^\vee \cdot
(z^2 \partial_z)^\vee$. 

We'd like to bootstrap this to the global case. Consider the
filtration of $\Omega^*\left(\Sigma, C^*_{\rm red}({\rm
  Jet}_{\scr{L}^\Sigma}) \right)$ by form degree. This spectral
sequence has $E_2$-page
\ben
E_2 = \Omega^*\left(\Sigma, \ul{{\rm H}}^*\left(C^*_{\rm red}({\rm
  Jet}_{\scr{L}^\Sigma}\right)\right) .
\een
Here, $\ul{\rm H}$ means the cohomology $\scr{D}$-module. We have computed the cohomology of the fibers of $\ul{H}^*\left(C^*_{\rm red}({\rm
  Jet}_{\scr{L}^\Sigma})\right)$, and they are concentrated in a single
degree. Choosing a formal coordinate at a point in $\Sigma$
trivializes the fiber of this point to $\bfC\<\partial_z^\vee \cdot (z \partial_z)^\vee \cdot
(z^2 \partial_z)^\vee\>$. This trivialization is independent of
coordinate choice and compatible with the flat connection. Thus
\ben 
\ul{{\rm H}}^*(C^*_{\rm red}({\rm
  Jet}_{\scr{L}^\Sigma}) \simeq C^\infty_\Sigma[-3]
\een
with its usual flat connection. This completes the proof.
\end{proof}

\subsection{Symmetries by vector fields}

The primary appearance of the Virasoro vertex algebra in physics is as a
symmetry of two dimensional conformal field theories. That is, the
Virasoro vertex algebra acts on conformal field theories with a
specified central charge. Later on we will see an example of how the Virasoro
factorization algebra appears as a symmetry of certain holomorphic quantum field theories using the BV formalism as
developed in \cite{CG1}, \cite{CG2}. For now, we would like to discuss
the meaning of such a Virasoro symmetry on general holomorphic
factorization algebras. 

A vertex algebra is conformal of central
charge $c$ if there is an element $v_c \in V$ such that the Fourier
coefficients $L_n^V$ of the vertex operator $Y(v_c, z) = \sum_n L_n^V
z^{-n-2}$ span a Lie algebra that is isomorphic to $\Vir_c$. Moreover,
one requires that $L_{-1}^V = T$ the translation operator, and
$L_{0}^V|_{V_n} = n \cdot {\rm Id}_{V_n}$. This can be wrapped up by
saying we have a map of vertex algebras ${\bf Vir}_c \to V$ sending
$L_{-2} \cdot 1$ to $v_c$. 

Motivated by this, we introduce the following terminology for
holomorphic factorization algebras. We say a
Virasoro symmetry of central charge $c$ of a holomorphic factorization
algebra $\sF$ is a map of holomorphic factorization algebras
\be\label{vir sym}
\Phi : \sVir_c \to \sF .
\ee

\begin{remark} A holomorphic factorization algebra is a
  symmetric monoidal functor 
\ben
\sF : {\rm Hol}_1 \to \dgNuc 
\een
A map of holomorphic factorization algebras is a natural
transformation between symmetric monoidal functors of the above
form. In particular, in the definition above we require the existence
of maps 
\ben
\Phi(\Sigma) : \sVir_c(\Sigma) \to \sF(\Sigma)
\een
for each one-dimensional complex manifold $\Sigma$. Moreover, these
maps must be natural with respect to holomorphic embeddings. 
\end{remark}

In the next section we will show an example using BV quantization to
implement a map of factorization algebras $\sVir_c \to \sF$ where
$\sF$ is the quantum observables of the $\beta\gamma$ system. In the
remainder of this section we'd like to extract one consequence of
having a Virasoro symmetry of charge $c$. We see that in the case of
factorization algebras on $\bfC$ we recover the usual notion of a conformal vertex algebra.

Indeed, suppose that $\sF$ is a holomorphic factorization algebra
  on $\bfC$ satisfying the conditions of Theorem \ref{fv}. Then a
  Virasoro symmetry of central charge $c$ from (\ref{vir sym}) induces the
  structure of a conformal vertex algebra on $\VVert(\sF)$ of charge
  $c$. As the construction $\VVert(\sF)$ is functorial we
  obtain a map of vertex algebra
\ben
\VVert(\Phi) : \Vir_c \to \VVert(\sF). 
\een
Explicitly, the conformal vector is given by $\VVert(\Phi)(L_{-2} 1_{\rm
    Vir}) \in \VVert(\sF)$. 

A fundamental object in conformal field theory is the so-called bundle of conformal blocks on
the moduli space of curves.  Given a vertex algebra describing a holomorphic
conformal field theory the action of the Virasoro Lie aglebra is necessary for the
construction of bundle equipped with a projectively flat connection through a process called ``Virasoro
uniformization'', see Chapter 17 of \cite{FBZ}, for instance. This is
a version of Gelfand-Kazhdan descent (sometimes referred to as
Harish-Chandra localization) along a certain bundle of
coordinates over the pointed moduli of curves
$\mathscr{M}_{g,1}$ which we briefly summarize. The moduli space
$\mathscr{M}_{g,1}$ consists of pairs $(\Sigma, x)$ where $\Sigma$ is a
curve and $x \in \Sigma$. There exists a canonically defined
$\Vir_0$-torsor $\Hat{\mathscr{M}}_{g,1}$ over $\mathscr{M}_{g,1}$
consisting of triples $(\Sigma,x,\varphi)$ where $\varphi$ is a formal
coordinate near $x$. One then considers modules for the {\em pair}
$({\rm Aut}(\hat{D}), \Vir_0)$ where $\Aut(\Hat{D})$ is the group automorphisms
of the formal disk. These objects are simultaneously modules for
$\Vir_0$ and the group ${\rm Aut}(\Hat{D})$ that are compatible
with the natural inclusion of Lie aglebras ${\rm Lie}({\rm Aut}(\Hat{D}))
\hookrightarrow \Vir_0$. In practice, and in all of the examples we
care about, $\Vir_0$-modules can be exponentiated to modules for the
pair. 

If one starts with a module $V$ for the pair $({\rm Aut}(\hat{D}),
\Vir_0)$ Virasoro uniformization can be viewed as a two-step
process. First, one forms the associated bundle over
$\mathscr{M}_{g,1}$ using the action of formal automorphisms. The
residual action of $\Vir_0$ defines the data of a flat connection. In
the case that one has an action of $\Vir_c$, for some nonzero charge
$c$, one gets a {\em projectively flat} connection. The resulting
object is no longer a D-module on the moduli of curves, but rather a
module for a sheaf of twisted differential operators. In
the case that $V$ is a conformal vertex algebra, the resulting bundle
is the bundle of conformal blocks equipped with its projectively flat
connection. For instance, in the case of the Virasoro vertex algebra
of central $c$, one finds a sheaf of twisted differential operators
on $\scr{M}_{g,1}$ (see \cite{BS}, for instance). 

One can attempt a similar construction at the level of factorization
algebras. Indeed, let $\sF$ be a holomorphic factorization algebra on
$\bfC$ that is equipped with a map of factorization algebras
\ben
\Phi : \sVir_c \to \sF
\een 
as in (\ref{vir sym}). In the case that $\sF$ is holomorphic we see
that $\sF(D)$ is a module for the Lie algebra of annular
observables. Indeed, we have a factorization map
\ben
\mu : \sF(A) \tensor \sF(D) \to \sF(D_{big})
\een
where $D_{big}$ is a disk centered at zero containing the annulus
$A$. We have already seen that the structure maps coming from nested
annuli give $\sF(A)$ the structure of a Lie algebra. Suppose that for
any inclusions of disks centered at zero
  $D(0,r) \hookrightarrow D(0,R)$ the induced map
\ben
\sF(D(0,r)) \xrightarrow{\simeq} \sF(D(0,R))
\een
is a quasi-isomorphism. Then, the structure map $\mu$ together with
the map ${\rm H}^*\Phi(A) : \Vir_c \simeq H^*(\sVir_c(A)) \to \sF(A)$ give ${\rm
  H}^*(\sF)$ the structure of a module over the Lie algebra ${\rm
  Vir}_c$. Thus, we can descend the space $\sF(D)$ to get a sheaf equipped with a
projective flat connection on $\mathscr{M}_{g,1}$. One expects that
the fiber of this bundle over a fixed curve $\Sigma$ coincides with
the global sections, or factorization homology $\int_\Sigma \sF$
defined in the next section. 

\section{Factorization homology and correlation functions}

\subsection{Global sections}

In this section we compute the cohomology of the global sections of the factorization
algebra $\sVir^\Sigma$. This is known as the factorization
homology of $\sVir^\Sigma$ and is denoted by
\ben
\int_\Sigma \sVir^\Sigma = {\rm H}^*(\sVir^\Sigma (\Sigma)) .
\een
In the language of chiral algebras the cohomology of global sections is often referred to as the ``chiral homology'' in the literature
\cite{FG, BD} and is dual to the space of ``conformal blocks''. We
will discuss conformal blocks for the Virasoro in more detail shortly.  

\begin{remark} 
As noted above ${\rm B}\scr{L}^\Sigma$ describes the formal completion at 
$\Sigma$ inside of $\scr{M}_g$, the moduli of Riemann surfaces of
genus $g$. We have already remarked that $\int_\Sigma \sVir^\Sigma$ is
the $\infty$-jet at $\Sigma$ of a certain sheaf on the
moduli of curves $\scr{M}_g$; namely the
sheaf of twisted differential operators. An independent definition of this
sheaf on the moduli of curves from the point of view of factorization
algebras is non-trivial, and we defer making any
precise relationships at the moment.
\end{remark}

Now, we compute the factorization homology. We will need the following
fact about the Dolbeault resolution of holomorphic vector fields. 

\begin{prop} The dg Lie algebra $(\Omega^{0,*}(\Sigma,T_\Sigma),
  \Bar{\partial})$ is formal. 
\end{prop}

That is, the dg Lie algebras $\Omega^{0,*}(\Sigma,T_\Sigma)$ and ${\rm
  H}_{\Bar{\partial}}^*(\Omega^{0,*}(\Sigma,T_\Sigma))$ are
quasi-isomorphic. It follows that ${\rm H}^*(\Vir^\Sigma(\Sigma) )$
is equal to the cohomology of the complex
\ben
\left( {\rm Sym}({\rm
    H}^*(\Sigma,T\Sigma) \oplus \bfC \cdot C), \dd_{\rm CE} \right)
\een
since $\Bar{\partial}$ kills the central term $C$. 

The full differential on $\sVir^\Sigma$ is $\dbar + \d_{Lie} +
\omega$ where $\d_{Lie}$ is the Chevalley-Eilenberg differential for
the Lie algebra $\Omega^{0,*}(\Sigma, T\Sigma)$ and $\omega$ is the
cocycle of Section \ref{sec fixed}.

\noindent {\bf The case $g = 0$}

We have ${\rm H}^*(\Sigma_0,T_{\Sigma_0}) \cong \mathfrak{s}
\ell_2(\bfC)$ generated by the vector fields $\partial_z,
z \partial_z,$ and $z^2 \partial_z$. For degree reasons the central
extension does not contribute to the Lie differential. Thus
\ben
\int_{\Sigma_0} \sVir^{\Sigma_0} \; \cong \; {\rm
  H}^{\rm Lie}_*(\mathfrak{s}\ell_2(\bfC)) [c] \; \cong \; \bfC[y,C] 
\een
with ${\rm deg}(y) = 3$ and ${\rm deg}(C) = 0$.
\\ \\
\noindent {\bf The case $g=1$}

In this case we know that the dg Lie algebra ${\rm H}^*(\Sigma_1,T_{\Sigma_1}) = \bfC
\oplus \bfC[-1]$ with zero Lie bracket and zero
differential. Moreover, ${\rm H}^0$ is generated by the constant
vector field $\partial_z$. The bilinear form defining the
central extension vanishes on constant vector fields so doesn't
contribute to the Lie differential. Thus
\ben
\int_{\Sigma_1} \sVir^{\Sigma_1}\; \cong \; {\rm Sym}
(\bfC[1] \oplus \bfC \oplus \bfC\cdot C) \; \cong \; \bfC[x,y,C]
\een 
with ${\rm deg}(x) = -1$ and ${\rm deg}(y) = {\rm deg}(C) = 0$.
\\ \\
\noindent {\bf The case $g > 1$}

The dg Lie algebra is ${\rm H}^*(\Sigma_g, T_{\Sigma_g}) = \bfC^{3g-3}[-1]$. For degree
reasons this algebra is abelian and does not interact with the central
extension. Thus 
\ben
\int_{\Sigma_g} \sVir^{\Sigma_g} \; \cong \; {\rm Sym}
(\bfC^{3g-3} \oplus \bfC\cdot C) \; \cong \; \bfC[y_1,\ldots,y_{3g-3},C]
\een
with ${\rm deg}(y_1) = \cdots = {\rm deg}(y_{3g-3}) = {\rm deg}(c) =
0$.

\subsection{Correlation functions}

In this section we compute the {\em correlation functions} associted
to the Virasoro factorization algebra. These calculations are
reminiscent for those of the conformal blocks of a conformal vertex
algebra and exhibits the utility of our approach to factorization in CFT. 

Fix a Riemann surface $\Sigma$ and consider a collection of disjoint
opens $U_1,\ldots,U_n \subset \Sigma$. The $n$-point correlation
function for associated to these open sets is the factorization structure
map 
\ben
\Phi_{U_1,\ldots,U_N} : \sVir^{\Sigma}(U_1) \tensor \cdots \tensor 
\sVir^{\Sigma} (U_n) \to 
\sVir^{\Sigma}(\Sigma) .
\een

Consider the case of $\Sigma = \bfC$ and suppose that each of the opens $U_i$ is a biholomorphic to a disk of a
certain fixed radius $r$. Suppose, moreover, that $\sF$ is a holomorphically
  translation-invariant factorization algebra on $\bfC$. Then it is an
  algebra over the co-operad $\Omega^{0,*}({\rm Disks})$. In particular, for each $nt$ we can think
of the $n$-point correlator as a holomorphic function on the space
\ben
{\rm Disks}_n(r) \simeq {\rm Conf}_n(\bfC) .
\een 
We now describe an explicit way of calculating these $n$-point
correlation functions that bears some resemblance to the standard method
of computing correlation functions in conformal field theory. 

First, we fix a partial inverse $\dbar^{-1}$ for the Dolbeault operator $\dbar$ for
the holomorphic tangent bundle $T_\Sigma$. This operator
vanishes on harmonic functions and 1-forms and is inverse to $\dbar$
on the complement to the space of harmonic functions and 1-forms. We
can construct it as follows. Let $G$ be a Green's function for the
$\dbar$ operator. It satisfies the equation
\ben
\dbar G = \omega_{diag}
\een
where $\omega_{diag}$ is the $(1,1)$-form on $\Sigma \times \Sigma$ that
is the volume element along the diagonal and zero elsewhere. Given $G$ we define the
operator $\dbar^{-1}$ via the formula
\ben
(\dbar^{-1} \varphi)(z) = \int_w G(z,w) \varphi(w) . 
\een

 
Let $a_1,\ldots, a_n \in \Omega^{0,*}(\Sigma, T \Sigma)$ be
$\dbar$ closed. We will
write down a general formula for the cohomology class of the
factorization product $a_1\cdots a_n$. Moreover, suppose that $a_1$ is in the
orthogonal complement to harmonic $(0,*)$ forms. Then consider the expression
\bestar
(\dbar + \d_{Lie} + \omega) \left((\dbar^{-1} a_1) a_2 \cdots
  a_n\right) &=
  & (\dbar \dbar^{-1} a_1) a_2 \cdots a_n + \sum_{j = 2}^n (-1)^{j +
    1} [\dbar^{-1} a_1, a_j] a_2 \cdots \Hat{a_j} \cdots a_n \\
& + & \sum_{j=2}^n (-1)^{j+1} \omega(\dbar^{-1} a_1, a_j) a_2 \cdots
\Hat{a_j} \cdots a_n .
\eestar
The first line follows from the fact that the only non trivial Lie
bracket involving the elements $a_1,\ldots,a_n$ is between $a_1$ and
$a_j$ for $j \ne 1$. The second line follows from the fact that the cocycle
$\omega$ is cohomologically degree one. 

Since the term on the left hand side is exact in the cochain complex $\sVir^\Sigma(\Sigma)$ we have at the level of cohomology
\be\label{n point}
\left\lfloor a_1 \cdots a_n \right\rfloor = \sum_{j = 2}^n (-1)^{j} \left\lfloor [\dbar^{-1} a_1,
a_j] a_2 \cdots \Hat{a_j} \cdots a_n \right\rfloor + \sum_{j=2}^n (-1)^{j}
\omega(\dbar^{-1} a_1, a_j) \left\lfloor a_2 \cdots \Hat{a_j} \cdots a_n\right\rfloor .
\ee
In particular, we see that $\lfloor a \rfloor = 0$ for any $a$. 

\subsubsection{Genus zero}
We can use this formula to recover well-known relations involving the
genus zero correlation functions. Fix a collection of points
$(x_1,\ldots,x_n) \in {\rm Conf}_n(\bfC P^1)$ and suppose $\epsilon >
0$ is such that the collection of disks $\{D(x_i, \epsilon)\}$ are
pairwise disjoint. Fix radial bump functions $f_{x_i}(z,\zbar) = f(r^2)$ for the disks
$D(x_i, \epsilon)$ and consider the $(0,1)$-forms $f_{x_i}(z,\zbar) \d
\zbar \in \Omega^{0,1}(D(x_i, \epsilon))$ which define the holomorphic vector
field valued forms
\ben
a_{x_i}(z,\zbar) := f_{x_i}(z,\zbar) \d
\zbar\partial_z  \in
\Omega^{0,1}\left(D(x_i, \epsilon), T \bfC P^1|_{D(x_i,
  \epsilon)}\right) \subset \sVir(D(x_i, \epsilon))
\een
on $D(x_i, \epsilon)$. On should think of $a_{x_i}$ as a mollified version of a point-like
observable supported at $x_i$. 

We will compute the resulting $n$-point correlation functions
\[
\left\lfloor a_{x_1} \cdots a_{x_n} \right\rfloor \in {\rm H}^0 \int_{\bfC
  P^1} \sVir \cong \bfC \cdot C .
\]
Here, we note that each $a_{x_i}$ is a (linear) degree zero element in
the factorization algebra
$\sVir$ so the resulting element in factorization homology is also
degree zero. We have already computed that ${\rm H}^0$ of the
factorization homology on $\bfC P^1$ is one-dimensional spanned by
the central element $C$. 

Using the explicit form of the operator $\dbar^{-1}$ on $\bfC P^1$ we
find
\[
\dbar^{-1}(a_{x_i}(z,\zbar)) = \frac{1}{z-x_1} \partial_z .
\] 
For $a_i = a_{x_i}$, the recursive equation for the $n$-point function
Equation (\ref{n point}) becomes 
\bestar
\left\lfloor a_{x_1} \cdots a_{x_n} \right\rfloor & = & \sum_{j=2}^{n} (-1)^j
\left\lfloor \left[\frac{1}{z-x_1} \partial_z, a_{x_j}(z,\zbar)
  \right] a_{x_2} \cdots \Hat{a}_j \cdots a_{x_n} \right\rfloor \\ & + &
c \sum_{j=2}^n (-1)^j \omega \left(\frac{1}{z-x_1} \partial_z, f_j(z,\zbar) \d
  \zbar \partial_z \right) \left\lfloor a_{x_2}\cdots \Hat{a}_{x_j} \cdots a_{x_n}
\right\rfloor .
\eestar
Let us use this formula to compute the $n$-point function for small
$n$. We have already remarked that $\lfloor a_{x_1} \rfloor =0$. Now,
suppose $x_1 \ne x_2$, then the recursive formula implies
\ben
\left\lfloor a_{x_1} a_{x_2} \right\rfloor = c \omega \left(
  \dbar^{-1} a_{x_1}, a_{x_2} \right)  .
\een 
By definition of the cocycle $\omega$, the right-hand side is equal to
\ben
c \cdot \frac{1}{12} \int_{z}
\frac{1}{z-x_1} \partial_z^3(f_{x_2}(z,\zbar)) \d z \d \zbar .
\een 
Iterative application of integration by parts together with the fact that $\int
\varphi(z) f_{x_2}(z,\zbar) \d z \d \zbar = \varphi(x_2)$ yields
\ben
\left\lfloor a_{x_1} a_{x_2} \right\rfloor = \frac{c}{2} \frac{1}{(x_1
  - x_2)^4} .
\een 

We can compute $\lfloor a_{x_1} a_{x_2} a_{x_3}\rfloor$ in a similar
way. Since $\lfloor a_{x_i} \rfloor = 0$ the recursive formula implies
\be\label{3point}
\left\lfloor a_{x_1} a_{x_2} a_{x_3}\right\rfloor = \left \lfloor \left[
  \frac{1}{z-x_1} \partial_z, a_{x_2}(z,\zbar) \right] \cdot a_{x_3}
\right\rfloor - \left \lfloor \left[
  \frac{1}{z-x_1} \partial_z, a_{x_3}(z,\zbar) \right] \cdot a_{x_2} \right\rfloor .
\ee
Consider the first term above. We compute the Lie bracket
\ben
\left[\frac{1}{z-x_1} \partial_z, a_{x_2}\right] =
\frac{1}{z-x_1} \partial_z(f_{x_2}(z,\zbar)) \d \zbar \partial_z +
\frac{1}{(z-x_1)^2} f_{x_2}(z,\zbar) \d \zbar \partial_z .
\een 
Applying $\dbar^{-1}$ to this expression yields the vector field
\ben
\left(-\frac{1}{(z-x_2)^2(x_2-x_1)} +
  \frac{2}{(z-x_2)(x_2-x_1)^2}\right) \partial_z .
\een 
This calculation, combined with the fact that $\lfloor a \cdot b \rfloor = c
\omega(\dbar^{-1} a, b)$ yields 
\bestar
\left \lfloor \left[
  \frac{1}{z-x_1} \partial_z, a_{x_2}(z,\zbar) \right] \cdot a_{x_3}
\right\rfloor & = & - c
\omega\left(\frac{1}{(z-x_2)^2(x_2-x_1)} \partial_z, a_{x_3}(z,\zbar)
\right) + 2 c \omega\left(\frac{1}{(z-x_2)(x_2-x_1)^2}\partial_z,
  a_{x_3}(z,\zbar) \right) \\ & = & - \frac{c}{12} \int_z
\frac{1}{(z-x_2)^2(x_2-x_1)} \partial_z^3(f_{x_3}(z,\zbar)) \d z \d
\zbar + \frac{c}{6} \int_z
\frac{1}{(z-x_2)(x_2-x_1)^2} \partial_z^3(f_{x_3}(z,\zbar)) \d z
\d\zbar \\ & = & \frac{c}{(x_3-x_2)^4(x_2-x_1)}
\left(-\frac{2}{x_3-x_2} + \frac{1}{x_2-x_1} \right)
\eestar
The second term in (\ref{3point}) is obtained by sending $x_2
\leftrightarrow x_3$ in the above formula. In total, the sum is thus
\ben
\frac{c}{(x_3-x_2)^4(x_2-x-1)(x_3-x_1)} \left(-2
  \frac{x_3-x_2}{x_3-x_2} + \frac{x_3-x_1}{x_2-x_1} + 2
  \frac{x_2-x_1}{x_3-x_2} + \frac{x_2-x_1}{x_3-x_1}\right) .
\een
This simplifies to the following expression for the $3$-point correlator
\ben
\left\lfloor a_{x_1} a_{x_2} a_{x_3} \right\rfloor =
\frac{c}{(x_1-x_2)^2 (x_1-x_3)^2 (x_2-x_3)^2} .
\een 

For general $n$, the recursive formula implies that can write the $n$-point function
as
\bestar
\left\lfloor a_{x_1} \cdots a_{x_n} \right\rfloor & = & \sum_{j = 2}^n
\left(\frac{1}{x_j - x_1} \partial_{x_j} +
  \frac{1}{(x_j-x_1)^2}\right) \left\lfloor a_{x_2} \cdots a_{x_n} \right
\rfloor  \\ & + & \frac{c}{2} \sum_{j=2}^n (-1)^j \frac{1}{(x_j - x_1)^4} \left\lfloor a_{x_2} \cdots
  \Hat{a}_{x_j} \cdots a_{x_n} \right\rfloor .
\eestar 
This shows, in particular, that as a function on the space
${\rm Conf}_{n}(\bfC P^1)$ the correlation function is not only holomorphic, it
is rational. One can find this expression for the correlation
functions in the vertex algebra literature, see for instance Section 2 of
\cite{Zhu}. 

\section{Application: Virasoro symmetry for holomorphic factorization
  algebras}

\subsection{Example: the $\beta\gamma$ system}
\def\Obs{{\rm Obs}}
\def\sH{\mathscr{H}}
\def\Sym{{\rm Sym}}

In this section we'd like to explain an example of a family of holomorphic field
theories parametrized by an integer $n$ whose factorization algebra of observables has the structure of
a holomorphic factorization algebra that we denote $\Obs^q_n$ with a
Virasoro symmetry. We produce a map of
factorization algebras on $\bfC$ from the Virasoro factorization
algebra (at a certain central charge) to $\Obs^q_{n}$. As a corollary we show that we recover the usual Virasoro vector of
the $\beta\gamma$ vertex algebra. 

First we need to define the factorization algebra $\Obs^q_n$. First, we
define a precosheaf of dg Lie algebras. To a one-dimensional
Riemannian manifold $U$ we define the dg Lie algebra
\ben
\sH_n(U) := \Omega^{1,*}_c(U)^{\oplus n} \oplus
\Omega^{0,*}_c(U)^{\oplus n} \oplus \bfC [-1]
\een
with bracket given by
\ben 
[\varphi,\psi] := \sum_{i=1}^n \int_U \varphi_i \wedge \psi_i .
\een
Here we write each element in components as $\varphi = (\varphi_1,\ldots,\varphi_n) \in
\Omega^{*,*}_c(U)^{\oplus n}$ so that $\varphi_i \in
\Omega^{*,*}_c(U)$. 
The factorization algebra is obtained in a similar way to the envelope
of a local Lie algebra. To an open set $U$ we define $\Obs^{q}_n(U) :=
{\rm C}^{\rm Lie}_*(\sH_n(U))$. Explicitly
\be\label{qobs1}
\Obs^q_n(U) := \left(\Sym\left(\Omega^{1,*}_c(U)^{\oplus n}[1] \oplus
  \Omega^{0,*}_c(U)^{\oplus n}[1] \right) , \dbar + \Delta \right) 
\ee
where $\Delta$ is the Chevalley-Eilenberg differential coming from the
Lie bracket. We restrict ourselves to factorization algebras on the Riemann surface
$\bfC$. 

It is shown in \cite{CG1}, that $\Obs^q_1$ is a
holomorphic factorization algebra whose associated vertex algebra is
isomorphic, to the one-dimensional
$\beta\gamma$ vertex algebra. Similarly, one has the following

\begin{theorem}[\cite{CG1} Theorem 5.3.3.2] The vertex algebra
  $\VVert(\Obs^q_n)$ is isomorphic to the $n$-dimensional
  $\beta\gamma$ vertex algebra $V_n$. The vertex algebra $V_n$ has
  state space spanned by vectors
  $\{b_l^i, c_m^j\}$ where $l < 0$, $m \leq 0$ and $1 \leq i,j \leq
  n$ with vertex operators
\bestar
Y(b_{-1}^i, z) & = & \sum_{l < 0} b_m^i z^{-1-n} + \sum_{l \geq 0}
\frac{\partial}{\partial c_{-l}^i} z^{-1-n} \\
Y(c_{0}^j, z) & = & \sum_{m \leq 0} c^j_m z^{-m} - \sum_{m > 0}
\frac{\partial}{\partial b_{-m}^j} z^{-m} .
\eestar 
\end{theorem} 

\begin{prop}\label{virtobg} There is a map of factorization algebras on $\bfC$ 
\ben
\Phi : \sVir^{\bfC}_{c = n} \to \Obs^q_n
\een 
commuting with the $S^1$ action. This map quantizes the map of
factorization algebras
\ben
\Phi^{cl} : \sVir^{\bfC}_{c=0} \to \Obs^{cl}_n .
\een
\end{prop}
In particular, the map of factorization algebras produces a map of vertex algebras
\ben
\VVert(\Phi) : \VVert(\sVir^{\bfC}_{c=n}) \to \VVert(\Obs^q_n)
.
\een
Concluding the proof of the proposition we will see explicitly that the map of vertex algebras produced by the
above proposition recovers the usual conformal vector of the
$\beta\gamma$ vertex algebra. 

\subsection{Proof of Proposition \label{virtobg}}

The proof is based on an explicit calculation in terms of Feynman
diagrams in a version of renormalization developed in
\cite{C1} and \cite{CG2}. 

To describe the map in Proposition \ref{virtobg} It is necessary to
describe the factorization algebra $\Obs^q_n$ in terms of an effective
family of factorization algebras and functionals as in \cite{CG2}. The
general formalism starts with a classical field theory defined by a
symplectic form of cohomological degree $-1$ and produces from an
effective quantization, as in \cite{C1}, a factorization algebra of
quantum observables. 

The fields of the theory are
\ben
\sE_n := \Omega^{0,*}(\bfC)^{\oplus n} \oplus
\Omega^{1,*}(\bfC)^{\oplus n}  .
\een
We write the fields as $(\gamma, \beta)$ (hence the name), and the
components as $\gamma = (\gamma_1,\ldots, \gamma_n)$, $\beta =
(\beta_1,\ldots,\beta_n)$. The symplectic pairing is
\ben
\<\gamma, \beta\> = \sum_{i=1}^n \int_\bfC \gamma_i \beta_i
\een 
which is easily seen to have cohomological degeee $-1$. With this
pairing, we can express $\sE_n$ as $\Omega^{0,*}(\bfC) \tensor V
\oplus \Omega^{1,*}(\bfC) \tensor V^*$ where $V$ is a complex
$n$-dimensional vector space. The pairing comes from the dual
pairing on $V$. 

The classical observables supported on $\Sigma$ are simply the space
of algebraic functions on the space of fields. Keeping track of the
right notion of duals, for any open $U \subset \bfC$ we define
\ben
\Bar{\Obs}^{cl}_n (\bfC)  = \Sym
\left(\Bar{\Omega}^{1,*}_c(\bfC)^{\oplus n} \oplus
  \Bar{\Omega}^{0,*}_c(\bfC)^{\oplus n} \right) .
\een 
One checks immediately that this construction defines a factorization
algebra on $\bfC$.

The classical action of holomorphic vector fields is a very natural
one. Given any element $\alpha \in \Omega^{*,*}(U)$ and any section
of the holomorphic tangent bundle $X \in \Gamma(U, T^{1,0}U)$ we
define
\ben
X \cdot \alpha = L_X \alpha
\een
where $L_X \alpha$ denotes the Lie derivative of $\alpha$ by $X$. This definition
naturally extends to elements $X \in \sL^\bfC (U) = \Omega^{0,*}(U,
T^{1,0}U)$. This action of $\sL^\bfC$ on forms leads to an action of the
factorization algebra $\Bar{\Obs}^{cl}_n$ as follows. For $X \in
\sL^\bfC$ define the holomorphically translation invariant local functional $I^{\sL}_X \in \oloc(\sE_n)$ by
\ben
I^{\sL}_X(\gamma,\beta) = \int \beta \wedge (X \cdot \gamma) .
\een 
Note that this local functional is of cohomological degree $-1$ and so
we have described a map
\ben
I^{\sL} : \sL^{\bfC} \to \oloc(\sE_n)[-1] .
\een
The space of local functionals shifted up by one $\oloc(\sE_n)[-1]$ is
itself a dg Lie algebra with Lie bracket given by the Poisson bracket
$\{-,-\}$ induced from the pairing $\<-,-\>$. It is immediate to check
that $I^{\sL}$ is compatible with this bracket and hence defines a map
of local Lie algebras on $\bfC$. This implies that $I^{\sL}$
determines a Maurer-Cartan element of ${\rm C}_{\rm loc}^*(\sL^\bfC ;
\oloc(\sE_n))[-1])$ which we think of as encoding the action of
holomorphic vector fields on the classical field theory. 

A translation invariant local functional determines a classical
observable supported on any open set in $\bfC$. For each $U \subset
\bfC$ we then extend $I^{\sL}$ to a map of commutative dg
algebras $\Phi^{cl}(U): \Sym(\sL_c(U)[1]) \to
\Bar{\Obs}^{cl}(U)$. These combine to give a map of factorization
algebras
\ben
\Phi^{cl} : \sVir_{0} \to \Bar{\Obs}^{cl} .
\een 

The naive BV-Laplacian $\Delta_0$ defined by contraction with the
integral kernel $K_0$ of the symplectic form $\<-,-\>$ is ill-defined on
$\Bar{\Obs}^{cl}_n$ as it involves pairing distributional
sections. This was solved in the above description by working with a
smaller class of observabels: one defines
\ben
\Obs^{cl}_n (\bfC) \subset \Bar{\Obs}_n^{cl}(\bfC)
\een
to be the subspace of non-distributional sections of the appropriate
vector bundles. Then, on $\Obs^{cl}_n$ the operator $\Delta_0$ {\em is}
well-defined. Note that when we equip $\Obs^{cl}_n(\bfC)$ with the differential $\dbar + \Delta_0$ we obtain $\Obs^q_n$ as defined in (\ref{qobs1}). It is
immediate that the factorization structures coincide. 

There is another solution that is necessary to describe the map in
Proposition \ref{virtobg} that involves mollifying the operator
$\Delta_0$ to a family of operators $\Delta_L$ for each $L > 0$. This
approach is outlined in wide generality in Chapter 9 of \cite{CG2}. In
this example there is an obvious choice on how to mollify
$\Delta_0$. Let $\dbar^*$ be the Hodge dual operator to $\dbar$ with
respect to the Euclidean metric on $\bfC$. Then the commutator $[\dbar,
\dbar^*]$ is the Hodge laplacian. We let $K_{L,n}$ be the integral kernel
for the operator $e^{-L [\dbar,\dbar^*]}$. It is the unique graded
symmetric
element $K_{L,n} \in \sE_n \tensor \sE_n$, for $L > 0$, satisfying
\ben
\<K_{L,n}(z,w), \varphi(w)\>_w = (e^{-L [\dbar,\dbar^*]} \varphi)(z)
\een
for all $\varphi \in \sE_n$. Explicitly, one has $K_{L,n} = K_L
\tensor \left({\rm Id}_{V} + {\rm
    Id}_{V^*}\right)$ where 
\ben
K_{L} (z,w) = \frac{1}{4 \pi L} e^{- |z-w|^2/4 L} (\d \zbar \tensor 1 -
1 \tensor \d \Bar{w}) .
\een 
The mollified BV Laplacian is the operator $\Delta_L$ defined by
contraction with $K_L$. Note that this operator is well-defined on
$\Obs^{cl}_n$. The propagator of the theory is
$P_\epsilon^L(z,w) \tensor \left({\rm Id}_{V} + {\rm
    Id}_{V^*}\right)$ where 
\ben
P_\epsilon^L (z,w) = \int_{t = \epsilon}^L \frac{1}{16 \pi t}
e^{-|z-w|/4t} \d t .
\een

The space of global quantum observables at scale $L$ is the complex
\ben
\Bar{\Obs}^q_n(\bfC)[L] := \left(\Sym
\left(\Bar{\Omega}^{1,*}_c(\bfC)^{\oplus n} \oplus
  \Bar{\Omega}^{0,*}_c(\bfC)^{\oplus n} \right), \dbar +
\Delta_L \right) .
\een 
To get a factorization algebra structure we need to provide the space
of quantum observables supported on an arbitrary open $U \subset
\bfC$. This is more subtle than in the classical case since the
operator $\Delta_L$ has support everywhere on the complex line. In
fact, to have a reasonable definition we need to consider the BV
Laplcian for a more general class of parameterices. This is developed
fully in Chapter 8 of \cite{CG2}. We will not provide details here, as
the exact definition of the factorization structure will not be
used. The main result we will need is the following. 

\begin{prop} There is a quasi-isomorphism of factorization algebras on $\bfC$
\ben 
\Obs^{q}_n \xrightarrow{\simeq} \Bar{\Obs}^q_n
\een
where on the right hand side we use the effective BV quantization
provided by the regularized BV operator. 
\end{prop}

We are given a Maurer-Cartan element $I^{\sL}$ that encodes the action
of holomorphic vector fields on the classical factorization
algebra. Since the field theory underlying the factorization algebra
is free, the action lifts to an action of a shifted central extension of
holomorphic vector fields. This implies that we have a map of
factorization algebras
\be\label{noether}
\Phi : U_{\alpha} \sL \to \Bar{\Obs}^q_n
\ee
for some cocycle $\alpha \in {\rm C}_{loc}^*(\sL)$ parameterizing the
shifted central extension. In the language of effective quantization
of BV theories the cocycle $\alpha$ is the $L \to 0$ limit of the
obstruction $\alpha[L]$ of the one-loop quantum interaction
\ben
I^{\sL}[L] = \sum_{\substack{\Gamma \in \text{\rm Graphs of genus}
    \leq 1}} \frac{1}{|{\rm
    Aut}(\Gamma)|} W_\Gamma(P_0^L \tensor \left({\rm Id}_{V} + {\rm
    Id}_{V^*}\right) ; I^{\sL})
\een  
 to satisfy the quantum master equation:
\ben
\dbar I^{\sL}[L] \d_{\sL} I^{\sL}[L] + \frac{1}{2} \{I^{\sL}[L], I^{\sL}[L]\}_L +
\hbar \Delta_L I^{\sL}[L] =  \alpha [L] .
\een 
Here, $I^{\sL}[L] \in {\rm C}_{\rm Lie}^*(\sL ; \scr{O}(\sE_n))$ is
defined by homotopy RG-flow using the weight expansion in terms of
connected graphs of genus less than or equal to one. There is a subtle
point, the propogator $P_0^L$ is distributional by nature so a priori
the expression for $I^{\sL}[L]$ may not exist. The fact that
$I^{\sL}[L]$ is well-defined is a hallmark of holomorphic theories
having no counter terms when one uses the so-called ``chiral gauge''. 

With a calculation similar to that of Corollary 16.0.5 in \cite{WG2}
we have the following description of the effective obstruction cocycle
$\alpha[L]$. 

\begin{lemma}(\cite{WG2} Corollary 16.0.5))The obstruction $\alpha[L]$ is computed by the weight
\ben
\lim_{\epsilon \to 0} W_\Gamma (P_\epsilon^L \tensor \left({\rm Id}_{V} + {\rm
    Id}_{V^*}\right) , K_{\epsilon}\tensor \left({\rm Id}_{V} + {\rm
    Id}_{V^*}\right) ; I^{\sL})
\een
where $\Gamma$ is the one-loop connected wheel with two vertices. We
attach the propagator $P_\epsilon^L$ to one inner edge and $K_L$ to
the other inner edge.
\end{lemma}

With this lemma in hand we directly compute $\alpha[L]$. For $X =
f(z,\zbar) \partial_z$ and $g(z,\zbar) \d\zbar \partial_z$ in $\sL_c(\bfC)$
we have
\ben
\alpha[L] (f \partial_z, g \d \zbar \partial_z) = n \lim_{\epsilon \to
  0} \int_{\bfC_z
  \times \bfC_w} f(z,\zbar) \left(\partial_z P_\epsilon^L(z,w)\right)
g(z,\zbar) \d \zbar \left(\partial_w K_\epsilon(z,w)\right) .
\een 
The factor $n$ comes from the contraction of the tensors depending on
$V$. Next, we compute
\ben
\partial_{w} K_\epsilon (z,w) = \frac{1}{4 \pi \epsilon} \frac{\zbar
  - \Bar{w}}{4 \epsilon} e^{-|\zbar - \Bar{w}|^2 / 4\epsilon} .
\een
Similarly, 
\ben
\partial_z P_\epsilon^L (z,w) = \int_{t = \epsilon}^L \frac{1}{4 \pi
  t} \frac{(\zbar - \Bar{w})^2}{(4 t)^2} e^{-|\zbar - \Bar{w}|^2 /
  4t} \d t .
\een
After making the change of coordinates $(y = z - w, w)$, and plugging
in the expressions above we obtain an expression for the integral
inside the $\epsilon \to 0$ limit
\be\label{int1}
\frac{1}{16 \pi^2} \int_{\bfC_y \times \bfC_w} fg \d^2 y \d^2 w \int_{t = \epsilon}^L 
\frac{1}{\epsilon t} \frac{1}{(4 \epsilon)(4 t)^2} \Bar{y}^3 \exp
\left( -\frac{1}{4} \left(\frac{1}{t} + \frac{1}{\epsilon}\right)
  |y|^2 \right) .
\ee
If $\varphi$ is any compactly supported function then integration by
parts yields the relation
\ben
\int_y \varphi(y) \Bar{y}^k e^{- a |y|^2} \d^2 y = \frac{1}{a^k}
\int_{y} \left(\partial_y^3 \varphi\right)(y) e^{-a |y|^2} \d^2 y .
\een 
Applying this to the integral in (\ref{int1}) we obtain
\ben
\alpha[L] (f \partial_z, g \d \Bar{w} \partial_w) = n \lim_{\epsilon
  \to 0} \frac{1}{16 \pi^2}
\int_{\bfC_y \times \bfC_w} \partial^3_y(f g) (y,w) \int_{t = \epsilon}^L \frac{\epsilon}{(\epsilon + t)^3}
\exp \left( -\frac{1}{4} \left(\frac{1}{t} + \frac{1}{\epsilon}\right)
\right) .
\een 
Finally, performing integration in the $y$-direction using Wick's
formula, the right-hand side becomes
\ben
n \frac{1}{2 \pi} \left(\int_{\bfC_w} \partial_w^3 f g \d^2 w\right)
\lim_{\epsilon \to 0} \int_{t=\epsilon}^L
\frac{\epsilon^2 t}{(\epsilon+t)^4} \d t  .
\een 
The $t$-integral converges and in the $\epsilon \to 0$ limit 
\ben
\int_{t=\epsilon}^L
\frac{\epsilon^2 t}{(\epsilon+t)^4} \xrightarrow{\epsilon \to 0}
\frac{1}{12} .
\een 
Note that there is no longer a dependence on the $L > 0$
parameter. This means that for any $L$ the functional $\alpha = \alpha[L]$ is
already a local functional representing the shifted central extension. In conclusion we have calculated
\ben
\alpha (f \partial_z, g \d \zbar \partial_z) = \frac{1}{2 \pi}
\frac{n}{12} \int_{\bfC_z} \partial_z^3 f g \d^2 z .
\een 
This is precisely the defining cocycle for the Virasoro
factorization algebra of charge $c = n$. In conclusion we see that the map of
factorization algebras (\ref{noether}) becomes
\ben
\Phi : \sVir_{c = n} \to \Obs^q_n
\een 
 as desired.

\bibliographystyle{alpha}
\bibliography{vir}

\end{document}